\documentclass[12pt,centertags,oneside]{amsart}
\usepackage{amsmath,amstext,amsthm,amscd,typearea}
\usepackage{amssymb}
\usepackage{a4wide}
\usepackage[mathscr]{eucal}
\usepackage{mathrsfs}
\usepackage{typearea}
\usepackage{charter}
\usepackage{pdfsync}
\usepackage[a4paper,width=16.2cm,top=3cm,bottom=3cm]{geometry}

\usepackage{url}  

\numberwithin{equation}{section}

\newtheorem{theorem}{Theorem}[section]
\newtheorem{definition}[theorem]{Definition}
\newtheorem{proposition}[theorem]{Proposition}
\newtheorem{corollary}[theorem]{Corollary}
\newtheorem{lemma}[theorem]{Lemma}
\newtheorem{remark}[theorem]{Remark}

\newcommand{\supp}{{\rm supp}}

\newcommand{\dist}{\mathop{\mathrm{dist}}\nolimits}
\renewcommand{\Re}{\mathop{\mathrm{Re}}\nolimits}
\renewcommand{\Im}{\mathop{\mathrm{Im}}\nolimits}

\def\Ker{\operatorname{Ker}}

\newcommand{\Id}{{\rm Id}}

\newcommand{\C}{\mathbb{C}}
\newcommand{\D}{\mathbb{D}}
\newcommand{\N}{\mathbb{N}}

\newcommand{\R}{\mathbb{R}}
\renewcommand\P{\mathbb{P}}

\newcommand{\Sb}{\mathbb{S}}

\renewcommand{\H}{\mathbb{ H}}

\newcommand{\boldsym}[1]{\boldsymbol{#1}}

\title[Fekete points on some real manifolds  ]{Equidistribution rate for Fekete points on some  real manifolds}

\author{Duc-Viet Vu}
\address{UPMC Univ Paris 06, UMR 7586, Institut de
Math{\'e}matiques de Jussieu-Paris Rive Gauche, 4 place Jussieu, F-75005 Paris, France.}
\email{duc-viet.vu@imj-prg.fr}
\thanks{This research is supported by grants from R\'egion Ile-de-France. }

\date{\today }
\begin{document}

\begin{abstract}  Let $L$ be a positive line bundle over a compact complex projective manifold $X$ and $K \subset X$ be a compact set which is regular in a sense of pluripotential theory. A Fekete configuration of order $k$ is a finite subset of $K$ maximizing a Vandermonde type determinant associated with the power $L^k$ of $L.$ Berman, Boucksom and Witt Nystr\"om proved that the empirical measure associated with a Fekete configuration converges to the equilibrium measure of $K$ as  $k \rightarrow \infty$. Dinh, Ma and Nguyen obtained an estimate for the rate of convergence. Using techniques from Cauchy-Riemann geometry, we show that the last result holds when $K$ is a real nondegenerate $\mathcal{C}^5$-piecewise submanifold of $X$ such that its tangent space at any regular point is not contained in a complex hyperplane of the tangent space of $X$ at that point. In particular, the estimate holds for Fekete points on some compact sets in $\R^n$ or the unit sphere in $\R^{n+1}.$
\end{abstract}

\maketitle

\medskip

\noindent
{\bf Classification AMS 2010}: 32U15.

\medskip


\tableofcontents

\section{Introduction}
The aim of this paper is to obtain an estimate on the rate of convergence of Fekete points on some compact sets toward the equilibrium state. In the view of possible applications, the class of compact sets that we consider is large enough and the required conditions for our compact sets are simple to check. Before introducing the general complex setting, let us discuss a simple but already important case of  Fekete points for a compact $K$ of the unit sphere $\Sb^n$ of $\R^{n+1}.$ 

For each $k\in \N,$ let $\mathcal{P}_k(K)$ be the real vector space of all real polynomials of degree at most $k$ in $(n+1)$ variables restricted to $K.$ Let $N_k$ be the dimension of  $\mathcal{P}_k(K).$ Given a basis $S=\{s_1,\cdots,s_{N_k}\}$ of $\mathcal{P}_k(K),$ consider the Vandermonde determinant $\det S$ of $S$ defined by assigning each point $x=(x_1,\cdots,x_{N_k}) \in (\Sb^n)^{N_k}$ to 
$$\det S(x):= \det[s_j(x_l)]_ {1\le  j,l \le N_k}.$$   
\emph{A Fekete point} of order $k$ for $K$ is a point $x \in K^{N_k}$ maximizing the absolute function $\lvert\det S \rvert$ on $K^{N_k}.$ It is easy to see that this definition does not depend on the choice of the basis $S$. The study of Fekete points is motivated by the fact that they are good choices of points for the interpolation problem of functions by polynomials, see, e.g., \cite{BosLevenbergWaldron, Sloan_Womersley} and references therein for more information. For any Fekete point $x=(x_1,\cdots,x_{N_k})$ of order $k,$  the probability measure $\delta_x$ on $K$, defined by 
$$\delta_{x}:= \frac{1}{N_k} \sum_{j=1}^{N_k} \delta_{x_j},$$
is called the  Fekete measure of order $k$ associated with $x.$  We are interested in the distribution of  Fekete points of order $k$ as $k \rightarrow \infty.$ A natural way to formulate this question is to study the limit points of Fekete measures in the space of probability measures on $K.$ 

Let $\mu_{eq}$ be the equilibrium measure of $(K,0)$ which is defined in Section \ref{sec_regularity}. When $K=\Sb^n,$ the measure $\mu_{eq}$ is simply the normalized volume form on $\Sb^n$  induced by the Euclidean metric on $\R^{n+1}.$ In this case,  J. Marzo and J. Ortega-Cerd\`a in \cite{Marzo_Ortega-Cerda_Equidistribution} prove that Fekete measures of order $k$ converge weakly to $\mu_{eq}$ as $k \rightarrow \infty.$ In general, a recent result of  R. Berman, S. Boucksom and D. Witt Nystr{\"o}m in \cite{BoucksomBermanWitt} shows that the weak convergence also holds for any compact $K$ of $\Sb^n$ which is non-pluripolar in the natural complexification $\Sb^n_{\C}$   of $\Sb^n.$ In fact, this result holds in a general setting of Fekete points
associated with a big line bundle over a compact complex manifold. Also in this general setting, Dinh, Ma and Nguyen \cite{DinhMaVietanh} introduced a notion of $(\mathcal{C}^{\alpha}, \mathcal{C}^{\alpha'})$-regularity and obtained a precise estimate
on the speed of convergence of Fekete points when the compact $K$ satisfies this property. We will show that such a property holds for  the closures $K$ of open subsets of $\Sb^n$ with nondegenerate piecewise smooth boundary (see Definition \ref{def_piecewisesmooth}).  As a consequence, we will  have the following.

 \begin{theorem} \label{the_sphere} Let $K$ be the closure of an open subset of $\Sb^n$ with nondegenerate piecewise smooth boundary. For every $\epsilon\in (0,1),$ there is a positive constant $c_{\epsilon}$ independent of $k\in \N$ such that for any Fekete measure $\mu_k$ of order $k$ of $K,$ we have  
 \begin{align} \label{speed_sphere}
\dist_{1}(\mu_k, \mu_{eq}) \le c_{\epsilon} k^{-1/72+ \epsilon}.
\end{align} 
\end{theorem}
Recall that for any two probability measures $\mu,\mu'$ on a compact differentiable manifold $X$ and a real number $\gamma>0,$ define
$$\dist_{\gamma}(\mu,\mu'):= \sup_{\|v\|_{\mathcal{C}^{\gamma}} \le 1} \big| \langle \mu-\mu', v \rangle \big|,$$
where $v$ is a smooth real-valued function. This distance induces the weak topology on the space of probability measures on $X.$ For two positive numbers $\gamma, \gamma'$ with $\gamma < \gamma',$ the distances $\dist_{\gamma}$ and $\dist_{\gamma'}$ are related by the inequalities
$$\dist_{\gamma'} \le \dist_{\gamma} \le c \dist_{\gamma'}^{\gamma/\gamma'},$$
for some positive constant $c,$ see \cite{DinhSibony_Pk_superpotential,Triebel}. Note that $\dist_1(\mu,\mu')$ is equivalent to the 
 Kantorovich-Wasserstein distance. We have a better estimate when $K=\Sb^n.$
 
\begin{theorem}\label{the_sphere2} For every $\epsilon\in (0,1),$ there is a positive constant $c_{\epsilon}$ depending only on $(n,\epsilon)$  such that for any Fekete measure $\mu_k$ of order $k$ of $\Sb^n,$ we have  
 \begin{align} \label{speed_sphere2}
\dist_{1}(\mu_k, \mu_{eq}) \le c_{\epsilon} k^{-1/36+ \epsilon}.
\end{align} 
 \end{theorem}
 It is  worth  mentioning also that when $K$ is the closure of an open subset of $\R^{n}$ with nondegenerate piecewise smooth boundary, one can define the Fekete points in $K$ and Fekete measures in the same way as above. The analogue of the inequality $(\ref{speed_sphere})$ also holds for this case. This is implied from our general result  by using the compact complexification $\P^{n}$ of $\R^{n},$ where $\P^{n}$ is the complex projective space of dimension $n.$ 


In order to prove Theorem \ref{the_sphere}, we will work in the following context of complex geometry.  Let $X$ be a compact $n$-dimensional complex projective manifold admitting  an ample line bundle $L$. Fix a  smooth Hermitian metric $h_0$ on $L$ with positive curvature. Let  $\omega_0$ be the first Chern form of $(L,h_0)$ which is a K\"ahler form on $X.$  For  $k \in \N,$ let $H^0(X, L^k)$ be the complex vector space of global holomorphic sections of $L^k.$ We also use $N_k$ to denote the dimension of $H^0(X,L^k)$. This will cause no ambiguity because we discuss essentially the general case from now on. Consider a basis 
$$S=(s_1,s_2,\cdots, s_{N_k})$$
 of $H^0(X,L^k)$ which can be seen to be a section of the  vector bundle $L^k \times \cdots \times L^k$ of rank $N_k$ over $X^{N_k}.$ The determinant  
$$\det S(p):= \det[s_i(p_j)]_{1\le i,j \le N_k}$$
 with $p=(p_1,\cdots, p_{N_k}) \in X^{N_k}$ defines a section of the determinant line bundle $L^{\boxtimes N_k}$ of the last bundle. The metric $h_0$ induces naturally a metric on $L^{\boxtimes N_k}$ . Denote by $|\det S|$ the norm of $\det S$ measured by this natural metric on $L^{\boxtimes N_k}.$  

Let $K$ be a compact subset of $X.$ Let $\phi$ be a continuous function on $K.$  \emph{The weighted Vandermonde determinant} $|\det S|_{\phi}$ at a point   $p=(p_1,\cdots, p_{N_k}) \in K^{N_k},$ by definition, is
$$|\det S|_{\phi}(p):= |\det S (p)| e^{-k \phi(p_1)- \cdots - k \phi(p_{N_k})}.$$  
 A \emph{Fekete configuration of order $k$} associated with $(K,\phi)$ and $(L,h_0)$  is a point in  $K^{N_k}$ maximizing the above weighted Vandermonde determinant on $K^{N_k}.$ 
The associated probability measure
$$\frac{1}{N_k} \sum_{j=1}^{N_k}\delta_{p_j}$$
 on $K$ is called a \emph{Fekete measure of order $k.$}  

Recall that a convex polyhedron $K''$ in $\R^{M}$ with $M \in \N^*$ is the intersection of a finite number of closed half-spaces in $\R^{M}.$ Its dimension  is defined to be  the one of the smallest vector subspace of $\R^{M}$ containing it. Such subspace is said to support $K''$. We define the complementary polyhedron of $K''$ to be the complement of $K''$ in the vector subspace supporting it. That complementary polyhedron is clearly a finite union of convex ones.  

\begin{definition}\label{def_piecewisesmooth}
A subset $K$ of a real $M$-dimensional smooth manifold $Y$ is called a nondegenerate $\mathcal{C}^5$-piecewise submanifold of dimension $m$ if for every point $p \in K,$ there exists a local chart $(W_p,\Psi)$ of $Y$ such that $\Psi$ is a $\mathcal{C}^5$-diffeomorphism from $W_p$ to the unit ball of $\R^{M}$ and  $\Psi(K \cap W_p)$ is the intersection with the unit ball  of a finite union of  convex polyhedra of the same dimension $m.$ In particular, when $\Psi(K \cap W_p)$ is the intersection with the unit ball of a convex polyhedron of dimension $m$ or of the complementary polyhedron of a convex one of dimension $m$, we say that $K$ is a  $\mathcal{C}^5$ submanifold of dimension $m$ with nondegenerate piecewise boundary.
\end{definition} 
Let  $K$ be a nondegenerate $\mathcal{C}^5$-piecewise submanifold of dimension $m$ of some smooth manifold $Y.$  \emph{A regular point} of $K$ is a point of $K$ such that there exists an open neighborhood of it diffeomorphic to an open subset of $\R^m.$ The regular part of $K$ is the set of regular points of $K.$  \emph{The singularity of $K$} is the complement of the regular part of $K$ in $K.$

Now let $K$ be a nondegenerate $\mathcal{C}^5$-piecewise submanifold of $X.$ Since $X$ is a complex manifold, its real tangent spaces have a natural complex structure. We say that  $K$ is \emph{generic} in the sense of Cauchy-Riemann geometry if  the tangent space  at any regular point $p$ of  $K$ does not contain in a complex hyperplane of the (real) tangent space at $p$ of $X.$ One can see without difficulty that the dimension of a generic $K$ is at least $n.$ Here is our second main result.

\begin{theorem} \label{the2} Let $\alpha$ be  a real number in $(0,1).$ Let  $K$ be a compact generic nondegenerate $\mathcal{C}^5$-piecewise submanifold of $X$. Let $\phi$ be a function of H\"older class $\mathcal{C}^{\alpha}$ on $K.$ Then for every $0 < \gamma <2,$  there is a constant $c>0$ such that for every integer $k\ge 1$ and for every Fekete measure $\mu_k$ of order $k$ associated with $(K,\phi),$ we have
\begin{align} \label{convergence1}
\dist_{\gamma}\big(\mu_k, \mu_{eq}(K,\phi)\big) \le c k^{-\beta \gamma} (\log k )^{3 \beta \gamma},
\end{align}
where $\mu_{eq}(K,\phi)$ is the equilibrium measure of $(K,\phi)$ (see Definition \ref{def_equilibriummeasure}) and $\beta= \alpha/(48+24 \alpha)$. When $K$ has no singularity, the constant $\beta$ can be chosen to be $\alpha/(24+12 \alpha).$
\end{theorem} 
In general, when $K$ is an arbitrary non-pluripolar compact subset of $X$ and $\phi \in \mathcal{C}^0(K),$ Boucksom, Berman and Witt Nystr\"om  in  \cite{BoucksomBermanWitt} proved that $\mu_k$ converges weakly to $\mu_{eq}$. Using a different technique,  Lev and Ortega-Cerd\`a in \cite{LevOrtegaCerda} obtained an optimal speed for the $\dist_1$ provided that $K=X$ and $\phi$ is smooth $\omega_0$-strictly p.s.h. and the metric $e^{-2\phi}h_0$ is strictly positive. Very recently, Dinh, Ma and Nguyen in \cite{DinhMaVietanh} introduced the notion of $(\mathcal{C}^{\alpha},\mathcal{C}^{\alpha'})$-regularity and proved an estimate for the rate of convergence for every compact $K$ satisfying this property, see Theorem \ref{the_theorem15DMN}. In particular, they showed that the closure of an open subset of $X$ with smooth boundary enjoys such regularity.  

In this work, we will prove that the compact $K$ in Theorem \ref{the2} satisfies the regularity mentioned above. Hence Theorem \ref{the2} will follow immediately. For the proof, we develop ideas from \cite[Th. 2.7]{DinhMaVietanh}. The key point is to construct families of analytic discs partly attached to $K$ in $X$ with useful properties. These families will allow us to reduce the question to the case of dimension one. Although there are plenty of works concerning families of analytics discs, it seems that there is no result which can be used directly for our problem. We will establish  a  Bishop-type equation and prove that it has a (unique) solution which suits our purposes. For the reader's convenience,  a self-contained proof will be given. The construction is inspired by the work of   Merker and Porten in \cite{MerkerPorten, MerkerPorten2}. We also underline that the case where the singularity of $K$ is nonempty requires much more sophisticated technical arguments than the case without singularity although the ideas used in the both situations are similar.   

In Section \ref{sec_regularity}, we prove the aforementioned regularity property of $K$ by admitting the existence of  special families of analytic discs whose proof is postponed until Section \ref{seclocal} and \ref{secgeneral}.   In Section \ref{seclocal}, we prove the existence of the above families of analytic discs in the simplest case  by constructing special analytic discs partly attached to $\R^n$ or $(\R^+)^n$  in $\C^n$. In Section \ref{secgeneral},  we show that the required families can be obtained as small deformations of the previous ones constructed in Section \ref{seclocal}. 
 
\vskip 0.5cm
\noindent
{\bf Acknowledgement.} The author  would like to thank Tien-Cuong Dinh for his valuable help during the preparation of this paper. He also would like to express his gratitude to  Junyan Cao and Jo\"el Merker for fruitful discussions.  

\section{$(\mathcal{C}^{\alpha},\mathcal{C}^{\alpha'})$-regularity of generic submanifolds} \label{sec_regularity}
We first recall some definitions. A function $\psi: X \rightarrow \R \cup \{-\infty\}$ is called \emph{quasi-plurisubharmonic} (quasi-p.s.h. for short) if it is locally the sum of a plurisubharmonic function and a smooth one. A quasi-p.s.h. function is called \emph{$\omega_0$-p.s.h.} if $dd^c \psi +\omega_0 \ge 0.$  Let $K$ be a compact subset of $X$ and $\phi$ be a real-valued continuous function on $K.$ The pair  $(K,\phi)$ is called \emph{a weighted compact subset} of $X$ and $\phi$ is called a \emph{weight} on $K.$  The \emph{equilibrium weight} associated with $(K,\phi)$ is the upper semi-continuous regularization $\phi_K^*$ of the function 
$$\phi_K:= \sup\{ \psi(z): \psi  \text{ $\omega_0$-p.s.h., } \psi \le \phi \text{ on $K$}\}.$$
Since the constant function $ - \max_K |\phi|$ is a $\omega_0$-p.s.h. and bounded above by $\phi$ on $K,$ we have $\phi_K \ge - \max_K|\phi|.$ It follows that $\phi_K$ is bounded from below. Recall that $K$ is said to be pluripolar if it is locally contained in $\{\psi=-\infty\}$ for some  (local) p.s.h. function $\psi,$ otherwise we say that $K$ is non-pluripolar. It is well-known that  $\phi_K$ is bounded from above if and only if $K$ is non-pluripolar. In this case, $\phi^*_K$ is a bounded $\omega_0$-p.s.h. function. The Monge-Amp\`ere measure $(\omega_0+ dd^c \phi^*_K)^n$  is hence well-defined. Its mass on $X$ equals $\int_X \omega_0^n$ by Stokes' theorem.   The \emph{equilibrium measure} of $(K,\phi)$ is the normalized Monge-Amp\`ere measure defined by 
\begin{align} \label{def_equilibriummeasure}
\mu_{eq}(K,\phi):= \frac{\big(dd^c(\phi_K^*)+\omega_0\big)^n}{\int_X \omega_0^n} \cdot
\end{align}
Recall that $\mu_{eq}(K,\phi)$ is a probability measure supported on $K.$  
When $K$ is an arbitrary compact generic nondegenerate $\mathcal{C}^5$-piecewise submanifold of $X$, it is a  direct consequence of Theorem \ref{the1} below that $K$ is non-pluripolar.

Fix a Riemannian metric on $X.$ For $p\in X$ and $r>0,$ let $B_X(p,r)$ be the ball centered at $p$ of radius $r$ of $X.$ Put  $B^*_X(p,r):= B_X(p,r)\backslash \{p\}.$  Recall that for $0 < \alpha <1,$ $\mathcal{C}^{\alpha}(X)$ is the space of real functions of H\"older class $\mathcal{C}^{\alpha}$ on $X$ with the norm defined by 
$$\|\phi\|_{\mathcal{C}^{\alpha}(X)}:= \sup_{p\in X} |\phi(p)|+  \sup_{p \not = p', p ,p'\in X} \frac{|\phi(p)- \phi(p')|}{\dist(p,p')^{\alpha}},$$
where $\dist$ denotes the distance on $X.$ The space $\mathcal{C}^{\alpha}(K)$ is defined similarly.

\begin{definition} For $\alpha \in (0,1)$ and $\alpha' \in (0,1),$ a  non-pluripolar compact $K$ is said to be $(\mathcal{C}^{\alpha}, \mathcal{C}^{\alpha'})$-regular if 
for any positive constant $C,$ the set $\{\phi_K: \phi \in \mathcal{C}^{\alpha}(K) \text{ and } \|\phi\|_{\mathcal{C}^{\alpha}(K)} \le C\}$ is a bounded subset of $\mathcal{C}^{\alpha'}(X).$
\end{definition}
By definition, if $K$ is $(\mathcal{C}^{\alpha}, \mathcal{C}^{\alpha'})$-regular, for any H\"older continuous function $\phi$ of order $\alpha>0$ on $K,$  we have $\phi^*_K= \phi_K$ because the latter is also H\"older continuous.  The notion of  $(\mathcal{C}^{\alpha}, \mathcal{C}^{\alpha'})$-regularity is essential in our work. The reason is the following result. 
\begin{theorem} \label{the_theorem15DMN} (\cite[Th. 1.5]{DinhMaVietanh}) Let $K$ be a non-pluripolar compact subset of $X.$ Let $\alpha \in (0,1),$ $\alpha' \in (0,1)$ and $\gamma \in (0,2].$ Assume that $K$ is $(\mathcal{C}^{\alpha}, \mathcal{C}^{\alpha'})$-regular. Let $\phi$ be a $\mathcal{C}^{\alpha}$ real-valued function on $K.$ Then, there is $c>0$ such that for every $k>1,$ we have 
$$\dist_{\gamma}\big(\mu_k, \mu_{eq}(K,\phi)\big) \le c k^{-\beta \gamma} (\log k)^{3 \beta \gamma},$$
with $\beta= \alpha'/ (24+12 \alpha').$
\end{theorem}
Theorem \ref{the2} is a direct  consequence of  Theorem \ref{the_theorem15DMN} and Theorem \ref{the1} below.

\begin{theorem} \label{the1} Let $\alpha$ be an arbitrary number in $(0,1).$ Then any compact generic nondegenerate $\mathcal{C}^5$-piecewise submanifold of $X$ is $(\mathcal{C}^{\alpha}, \mathcal{C}^{\alpha/2})$-regular. Particularly, when $K$ has no singularity,  $K$ is $(\mathcal{C}^{\alpha}, \mathcal{C}^{\alpha})$-regular.
\end{theorem}
When $K$ is of maximal real dimension, the regularity of $K$ can be improved,  see Remark \ref{re_dimK2n} for more details.

\begin{remark} \label{re_phanviduoptimal} Consider the case where $\dim_{\R} K=n.$ This is the case of our great interest. Then, the regularity of $K$ obtained in Theorem \ref{the1} is optimal. For simplicity, let take $n=1$ and $X=\P^1=\{[z_0:z_1]: (z_0,z_1)\in \C^2\backslash \{0\}\}.$  Let $\omega_0$ be the Fubini-Study form on $\P^1.$ Using the local coordinates $[z:1],$ $\P^1$ can be seen as the compactification of $\C$ with a point at infinity. An $\omega_0$-p.s.h. functions is equal to $\psi(z)- \frac{1}{2}\log(1+|z|^2)$ on $\C,$ where $\psi$ is a subharmonic function on $\C,$ such that the last difference is bounded above.  

Let $K=[-1,1].$ 
Choose $\phi \equiv 0$ on $K.$ Using \cite[Cor. 5.4.5]{Klimek}, we get $\phi_K(z)=\log |z+ \sqrt{z^2-1}|$ on $\C,$ where the square root is chosen such that 
$$ |z+ \sqrt{z^2-1}| \ge 1.$$
Comparing $\phi_K(z)$ with $0$ when $z$ is close to $1,$ one sees that $\phi_K \in \mathcal{C}^{1/2}(X) \backslash \mathcal{C}^{1/2 +\epsilon}(X)$ for any $\epsilon>0.$ In higher dimension, the same arguments also work for  $X= \P^n$, $K=[-1,1]^n \subset \C^n$ and $\phi \equiv 0.$   
\end{remark} 
 Before giving the proof of Theorem \ref{the1}, we need to recall the definition of analytic discs partly attached to a subset of $X$. Let $\D$ be the open unit disc in $\C.$ \emph{An analytic disc} $f$ in $X$ is a holomorphic mapping from $\D$ to $X$ which is continuous up to the boundary $\partial \D$ of $\D.$ For an interval $I \subset \partial \D,$ $f$ \emph{is said to be $I$-attached to a subset $E \subset X$} if $f(I) \subset E.$ In particular,  we say that $f$ is \emph{half-attached to  $E$} if  $f(\partial^+\D) \subset E,$ where $\partial^+\D= \{\xi \in  \partial \D: \Re \xi \ge 0\}.$  The crucial ingredient in the proof of Theorem \ref{the1} is Proposition \ref{pro_familydiscK} below which shows the existence of special families of analytic discs partly attached to $K$ in $X.$ Its proof will be given in Section \ref{secgeneral}. 

\begin{proposition}\label{pro_familydiscK}  There are  positive constants $c_0, r_0$ and $\theta_0 \in (0, \pi/2)$  such that for any $p_0 \in K$ and any $p\in B^*_X(p_0,r_0),$ there exist an open neighborhood $W_{p_0}$ of $p_0$ in $X$ independent of $p$ which is biholomorphic to the unit ball of $\C^n$ and a $\mathcal{C}^1$ analytic disc $f: \overline{\D} \rightarrow W_{p_0}$ such that $f$ is $[e^{-i\theta_0}, e^{i\theta_0}]$-attached to $K,$ $\dist(f(1),p_0)\le c_0 \delta$ with $\delta=\dist(p,p_0),$ $\|f\|_{\mathcal{C}^1} \le c_0$ and there is  $z^* \in \D$ so that $|1-z^*| \le \sqrt{c_0 \delta}$ and $f(z^*)=p.$ When $K$ has no singularity,  $z^*$ can be chosen so that  $|1-z^*| \le c_0\delta.$ 
\end{proposition}
We will also need the following lemma in complex dimension one.

\begin{lemma}\label{le_passagetogeneralcase3} Let $\theta_0 \in (0, \pi/2), \beta \in (0,1)$ and let $c>0$. Let $\psi$ be a subharmonic function on $\D$ and continuous on $\overline{\D}$.  Assume that 
\begin{align}\label{ine_giasutrenpsi}
\psi(e^{i\theta}) \le c|\theta|^{\beta} \quad \text{for} \quad \theta \in (-\theta_0,\theta_0) \quad \text{and} \quad  \sup_{\theta \in [-\pi,\pi]} \psi(e^{i\theta}) \le c.
\end{align}
Then, there exists a constant $C$ depending only on $(\theta_0,\beta,c)$ so that for any $z \in \D,$  we have
\begin{align*} 
\psi(z) \le  C |1-z|^{\beta}.
\end{align*}  
\end{lemma}

\begin{proof} Observe that the function $|\theta|^{\beta}$ is H\"older continuous of order $\beta$ in $\theta.$ By  using this fact and suitable cut-off functions in $\mathcal{C}^{\infty}(\partial \D),$ we see that there exists a function $\psi_1$ in $\mathcal{C}^{\beta}(\partial \D)$ so that 
$$\psi_1(e^{i\theta})= c|\theta|^{\beta} \quad \text{for} \quad \theta \in (-\theta_0/2, \theta_0/2)$$
 and 
$$ \psi_1 (e^{i\theta}) \ge c \quad \text{for} \quad \theta \in [-\pi,-3\theta_0/4] \cup [3\theta_0/4,\pi].$$
 By (\ref{ine_giasutrenpsi}), we have $\psi(e^{i\theta}) \le \psi_1(e^{i\theta})$ on $\partial\D.$   Extend $\psi_1$ harmonically to $\D.$ Denote also by $\psi_1$ its harmonic extension.  It is classical that $\psi_1 \in \mathcal{C}^{\beta}(\overline{\D}),$ see (\ref{ine_danhgiachuaCkcuauvoibien}) for details. Since $\psi$ is subharmonic on $\D$ and $\psi \le \psi_1$ on $\partial \D,$ we have $\psi \le \psi_1$ on $\D.$ As a result, 
$$\psi(z) \le \psi_1(z) \le \psi_1(1)+ \|\psi_1\|_{\mathcal{C}^{\beta}} |1-z|^{\beta},$$
 for any $z \in \D.$ The desired inequality follows because $\psi_1(1)=0.$  The proof is finished.
\end{proof}

\begin{proof}[Proof of Theorem \ref{the1}.] By \cite[Th. 2.7]{DinhMaVietanh}, the compact set $X$ is itself $(\mathcal{C}^{\beta}, \mathcal{C}^{\beta})$-regular for any $\beta \in (0,1)$.  Let $\tilde{\phi}$ be the function on $X$ defined by 
$$\tilde{\phi}(p):=\min_{p' \in K} [\phi(p')+ A \dist(p,p')^{\alpha/2}]$$
for $p\in X$ and $A \gg \|\phi\|_{\mathcal{C}^{\alpha}}$ is a fixed big constant.  It is not difficult to see that $\tilde{\phi}$ is $\mathcal{C}^{\alpha/2}$ and $\tilde{\phi}= \phi$ on $K.$ Namely, we have 
$$|\tilde{\phi}(p)- \tilde{\phi}(p')| \le A \dist(p,p')^{\alpha/2}$$
for all $p,p'\in X.$ By $(\mathcal{C}^{\alpha/2}, \mathcal{C}^{\alpha/2})$-regularity of $X,$ we have  $\tilde{\phi}_X=\tilde{\phi}_X^*$ which is $\omega_0$-p.s.h. and of H\"older class $\mathcal{C}^{\alpha/2}.$ Hence,  the proof of Theorem \ref{the1} is finished if we can show that 
\begin{align}\label{eq_phingaKX}
\phi_K^*= \tilde{\phi}_X.
\end{align}
 Clearly, by definition of $\phi_K^*$ and $\tilde{\phi}_X,$ we have  $\phi_K^* \ge \tilde{\phi}_X.$ Thus, to prove (\ref{eq_phingaKX}), it is enough to prove that 
$$\phi_K \le \tilde{\phi}$$
 on $X$ because this implies $\phi_K^* \le \tilde{\phi}$ thanks to the continuity of the last function. Since $A$ is big enough and $\phi_K$ is bounded on $X$, we only need to check that 
\begin{align} \label{ine_phiKnga}
\phi_K(p) \le \tilde{\phi}(p),
\end{align}
 for $p$ close to $K.$ This inequality is clear for $p\in K.$

 Fix a $p \not  \in K$ close to $K.$ Let $p_0$ be a point in  $K$ such that 
$$\dist(p,p_0)= \min_{p' \in K}\dist(p,p').$$ 
Put $\delta= \dist(p,p_0).$ It is shown in the proof of \cite[Th. 2.7]{DinhMaVietanh} that $\tilde{\phi}(p) \ge \tilde{\phi}(p_0)+ A' \delta^{\alpha/2}$ for some big constant $A'$ independent of $p$ ($A' \rightarrow \infty$ if $A \rightarrow \infty$). Hence, in order to prove (\ref{ine_phiKnga}), it suffices to prove that 
\begin{align}\label{ine_phiKpp0}
\phi_K(p) \le \tilde{\phi}(p_0)+ A' \delta^{\alpha/2}=\phi(p_0)+ A' \delta^{\alpha/2}.
\end{align}
As already said above, we only need to consider $\delta$ small.  Precisely, we will suppose that $\delta < r_0,$ where $r_0$ is  the constant in Proposition \ref{pro_familydiscK}. Let $f$, $W_{p_0}$,$c_0, \theta_0$ be the analytic disc corresponding to $(p_0,p)$, the open neighborhood of $p_0$ and the constants respectively in that proposition. There is  $z^* \in \D$ with  $|1-z^*| \le \sqrt{c_0\delta}$ so that $f(z^*)= p.$

 Let  $\psi$ be   an  $\omega_0$-p.s.h. function on $X$ so that $\psi \le \phi$ on $K.$ Since $W_{p_0}$ is biholomorphic to the unit ball of $\C^n,$ there exists  a smooth potential $\psi_{\omega_0}$ of $\omega_0$ on $W_{p_0}, i.e,$ we have  $$dd^c \psi_{\omega_0}=\omega_0 \quad  \text{ on} \quad W_{p_0}.$$
 Hence, $\psi_0:= \psi + \psi_{\omega_0}$ is a p.s.h. function on $W_{p_0}$ and $\psi_0 \le \phi_0:= \phi+ \psi_{\omega_0}$ on $W_{p_0} \cap K.$ By the smoothness of $\psi_{\omega_0},$ the function $\phi_0$ is also H\"older continuous of order $\alpha$ on any compact subset of $W_{p_0},$ hence on $\overline{f(\D)}.$  Define $\psi_1:= \psi_0 \circ f,$  and  $\phi_1:= \phi_0 \circ f.$   
Observe that  $\psi_1$ is a p.s.h. function on $\D$ and continuous on $\overline{\D}.$ We also have 
\begin{align}\label{equ_giatricuapsi1vaphi_1}
\psi_1(z^*)=\psi_0(p) \quad \text{and} \quad \phi_1(1)= \phi_0\big(f(1)\big). 
\end{align}
 Since $\|f\|_{\mathcal{C}^1} \le c_0$, the function $\phi_1$ is H\"older continuous of order $\alpha$ with a H\"older constant independent of $p,p_0$ and $f.$ On the other hand, since the disc $f$ is $[e^{-i\theta_0},e^{i\theta_0}]$-attached to $K,$ we have $\psi_1(e^{i\theta}) \le \phi_1(e^{i\theta})$ for $\theta \in [-\theta_0,\theta_0].$ This together with the H\"older continuity of $\phi_1$ yield that 
 \begin{align*} 
\psi_1(e^{i\theta}) \le \phi_1(1)+c |\theta|^{\alpha},  
\end{align*}
for $\theta \in [-\theta_0,\theta_0]$ and for some positive constant $c.$ 
Applying Lemma \ref{le_passagetogeneralcase3} to the subharmonic function $\big(\psi_1 - \phi_1(1)\big)$ gives 
\begin{align} \label{equ_giatricuapsi1vaphi_111}
\psi_1(z^*) \le \phi_1(1) + C |1-z^*|^{\alpha} \quad \text{ for some positive constant } \, C.
\end{align}
Combining (\ref{equ_giatricuapsi1vaphi_1}), (\ref{equ_giatricuapsi1vaphi_111}) and the definitions of $\psi_0, \phi_0,$ one obtains 
$$\psi(p) \le  \phi_0\big(f(1)\big)+ C |1-z^*|^{\alpha},$$
for every $\omega_0$-p.s.h. $\psi$ on $X$ with $\psi \le \phi$ on $K.$ Taking the supremum over all such $\psi$ in the last inequality and using the definition of $\phi_K$ give 
$$\phi_K(p) \le \phi_0\big(f(1)\big)+ C |1-z^*|^{\alpha} \le  \phi(p_0)+ \|\phi\|_{\mathcal{C}^{\alpha}} |1-z^*|^{\alpha}+ C  |1-z^*|^{\alpha} \le \phi(p_0)+ A' \delta^{\alpha/2}$$
because $\phi \in \mathcal{C}^{\alpha}$ and $|1-z^*| \le \sqrt{c_0 \delta}.$
Now consider the case  where $K$ has no singularity. Define
$$\tilde{\phi}'(p):=\min_{p' \in K} [\phi(p')+ A \dist(p,p')^{\alpha}]$$
for $p\in X$ and  some fixed big constant $A \gg \|\phi\|_{\mathcal{C}^{\alpha}}$. 
By using the same argument as above with $\tilde{\phi}'$ in place of $\tilde{\phi}$ and the fact that $|1-z^*| \le c_0 \delta,$  we get the desired conclusion.  The proof is finished. 
\end{proof}

\begin{proof}[ Proof of Theorems  \ref{the_sphere} and \ref{the_sphere2}]
We first prove Theorem \ref{the_sphere}. Recall that we have canonical inclusions: $\R^{n+1} \subset \C^{n+1} \subset \P^{n+1}.$ Let $\Sb^n_{\C}$ be the complexification of $\Sb^n$ in $\P^{n+1}$ defined by the equation 
$$z_0^2+ \cdots+ z_n^2= z_{n+1}^2,$$
where $[z_0: \cdots: z_{n+1}]$ are the homogeneous coordinates on $\P^{n+1}.$ We see that $\Sb^n$ is a compact generic submanifold of $\Sb^n_{\C}$. Choose $X= \Sb^n_{\C},$ $K=\Sb^n,$ $\phi=0$ and  $L= \mathcal{O}(1)|_{X}$ is the restriction of the hyperplane line bundle of $\P^{n+1}$ to $X.$ Observe that the restriction $H^0(X, L^k)|_{\Sb^n}$ of $H^0(X, L^k)$ to $\Sb^n$ is a complex vector space of (complex) dimension $\dim_{\R}\mathcal{P}_k(\Sb^n).$ As $K$ is non-pluripolar in $\Sb^n_{\C},$ any nonzero holomorphic function on an open subset of $\Sb^n_{\C}$ can not annihilate on the whole $K.$ As a result, we have  $\dim_{\R} \mathcal{P}_k(\Sb^n)= \dim_{\R} \mathcal{P}_k(K).$ This allows one to choose a common basis for the two vector spaces $H^0(X, L^k)|_{\Sb^n}$ and $\mathcal{P}_k(K)$ when defining Fekete points. Therefore,  Fekete points in the complex case are those defined on $K$ as in Introduction. Theorem \ref{speed_sphere} is now  a direct corollary of Theorem \ref{the2} with the choice of $(X,L,K, \phi)$ as above, $\gamma=1$ and $\alpha=1-\epsilon,$ for $\epsilon>0$.

Consider the case where $K= \Sb^n,$ the equilibrium measure $\mu_{eq}(K,0)$ coincides with the normalized volume form on $\Sb^n$ induced by the Euclidean metric on $\R^{n+1}$ because $\mu_{eq}(K,\phi)$ is preserved by the actions of the orthogonal matrix group on $\Sb^n.$ Theorem \ref{the_sphere2} is hence obtained in a  similar way by using the fact that $\Sb^n$ has no boundary.  The proof is finished. 
\end{proof}


\begin{remark} \label{re_dimK2n} We discuss here very briefly the case where $\dim_{\R} K=2n$ in which the regularity of $K$ can be improved. For simplicity, we consider the following simple model in the complex dimension $1.$ Let $X= \P^1= \C \cup \{\infty\}$ as in Remark \ref{re_phanviduoptimal}. Let  $K$ be the compact convex polygon in $\C.$ Denote by $S_1, S_2, \cdots,S_m$ the consecutive vertices of $K.$ Let $\pi<\gamma_j< 2\pi$ be the exterior angle at $S_j$ of $K,$ for $1\le j \le m.$  Put $\gamma= \max_{1\le j \le m}\gamma_j.$  Then, $K$ is $(\mathcal{C}^{\alpha}, \mathcal{C}^{\alpha \pi/ \gamma})$-regular. When $\gamma_j=\pi$ for all $j,$ we re-obtain  \cite[Th. 2.7]{DinhMaVietanh}.  The idea for the proof is as follows.  Let $\phi \in \mathcal{C}^{\alpha}(K).$ In order to get the above regularity of $K,$ it is enough to show that given any p.s.h. function $\psi$ on $\C$ so that $\psi \le \phi$ on $K,$ then for every $j,$ we have $\psi(z) \le \phi(S_j)+ A |z- S_j|^{\alpha \pi/\gamma}$ for every $z$ close to $S_j$ and for some fixed big constant $A.$ Let $L_1$ be the open domain of $\C$ limited by the two rays $S_1 S_{m}$ and $S_1S_{2}$ which does not contain $K.$ Using an affine change of coordinates, we can suppose that $S_1=0,$ the ray $S_1S_2$ is $\{z=(x,0): x\ge 0\}$ and $L_1 \supset \H:= \{ \Im z > 0\}.$ Using the map $(z-i)/(z+i)$ sending $\H$ biholomorphically to $\D,$ one easily sees  that the map $\Psi(z):=(z^{\pi/\gamma}-i)(z^{\pi/\gamma}+i)$ is a biholomorphism from $L_1$ to $\D.$ Clearly, $\Psi$ is H\"older continuous of order $\pi/\gamma$ on an open neighborhood of $S_1$ in $\overline{L}_1.$ An application of Lemma \ref{le_passagetogeneralcase3} to $\psi \circ \Psi^{-1}$ gives the desired result. 
\end{remark}

\section{ Two special families of analytic discs } \label{seclocal}
\subsection{Hilbert transform} \label{subsec_Hilberttransform}
Denote by $z=x+ i y$ the complex variable on $\C$ and by $\xi=e^{i\theta}$ the variable on $\partial \D.$ For any $m \in \N$ and $r >0,$ let $B_m(0,r)$ be the Euclidean ball centered at $0$ of radius $r$ of $\R^m$ and let $B_m^*(0,r)= B_m(0,r)\backslash \{0\}.$ Denote by $|\cdot|$ the Euclidean norm on $\R^m.$ The same notations will be used for $\C^n$ that we sometimes identify with $\R^{2n}.$ Let $Z$ be  a submanifold of $\R^m.$ The Euclidean metric on $\R^m$ induces a metric on $Z.$  For $\beta \in (0,1)$ and $k \in \N$, let $\mathcal{C}^{k,\beta}(Z)$ be the space of real-valued functions on $Z$ which are differentiable up to the order $k$ and whose $k^{th}$ derivatives are H\"older continuous of order $\beta.$ This is a Banach space with the  $\mathcal{C}^{k,\beta}$-norm given by
$$\|v\|_{k,\beta,Z}:=\|v\|_{k,Z}+ \sup_{\xi \not= \xi', \xi,\xi' \in Z} \frac{\|D^k v(\xi)- D^k v(\xi')\|}{|\xi - \xi'|^{\beta}},$$
where  $\|\cdot\|_{k,Z}:=\|\cdot\|_{\mathcal{C}^k(Z)}$ and $D^k v$ denotes the $k^{th}$-differential of $v.$  In the proof, we will only use this norm for $Z=\overline{\D}$ or $\partial \D.$ When $Z$ is clear from the context, we will remove the subscript $Z$ from the above notation of norm.   For any tuple $v=(v_0,\cdots,v_m)$ consisting of functions in  $\mathcal{C}^{k,\beta}(Z)$, we define its $\mathcal{C}^{k,\beta}$-norm to be the maximum of the ones of  its components.

Recall that an arbitrary continuous function $u_0(\xi)$ on $\partial \D$ can be extended uniquely to be a harmonic function on $\D$ which is continuous on $\overline{\D}$. Since this correspondence is bijective, without stating explicitly, we will freely identify $u_0$ with its harmonic extension on $\D.$ We will write $u_0(z)=u_0(x+iy)$ to indicate the harmonic extension of $u_0(e^{i\theta}).$ 
It is well-known  that the Cauchy transform of $u_0,$ given by 
$$\mathcal{C}u_0(z):= \frac{1}{2\pi} \int_{-\pi}^{\pi} u_0(e^{i\theta}) \frac{e^{i\theta}+ z}{e^{i\theta}-z} d\theta,$$
is a holomorphic function on $\D$ whose real part is $u_0.$ 
 Decomposing the last formula into the real and imaginary parts, we obtain that 
\begin{align} \label{equ_harmonicextension}
u_0(z)= \frac{1}{2 \pi} \int_{-\pi}^{\pi} \frac{(1- |z|^2)}{|e^{i\theta}-z|^2} u_0(e^{i\theta}) d\theta.
\end{align}
and 
\begin{align*} 
\mathcal{T}u_0(z)= \frac{1}{2 \pi } \int_{-\pi}^{\pi } \frac{(z e^{- i\theta}- \bar{z} e^{i\theta})}{i |e^{i\theta}-z|^2} u_0(e^{i\theta}) d\theta.
\end{align*}
The function $\mathcal{T} u_0$ is harmonic on $\D$ but is not always continuous up to the boundary of $\D.$ Let $k$ be an arbitrary natural number and let  $\beta$ be an arbitrary number in $(0,1).$ A result of Privalov (see \cite[Th. 4.12]{MerkerPorten} or \cite[Sec. 6.1]{BER}) implies that if $u_0$ belongs to  $\mathcal{C}^{k,\beta}(\partial\D)$, then $\mathcal{T}u_0$ is continuous up to $\partial \D$ and  $\|\mathcal{T}u_0\|_{k,\beta,\partial \D}$  is bounded by  $ \|u_0\|_{k,\beta,\partial \D}$ times a constant independent of $u_0.$ Hence, the linear self-operator of $\mathcal{C}^{k, \beta}(\partial \D)$ defined by sending $u_0$ to the  restriction of $\mathcal{T}u_0$ onto $\partial \D$ is bounded and called  \emph{the Hilbert transform}. For simplicity, we also denote it by $\mathcal{T}.$   In the method of analytic discs,  it is convenient to use a modified version $\mathcal{T}_1$ of $\mathcal{T}$ defined by
 $$\mathcal{T}_1 u_0:= \mathcal{T}u_0 - \mathcal{T}u_0(1).$$
Hence we  always have $\mathcal{T}_1 u_0(1)=0$ and
\begin{align}\label{eq_daohamcuaT}
\partial_{\theta}\mathcal{T}_1 u_0=\partial_{\theta} \mathcal{T} u_0= \mathcal{T}\partial_{\theta} u_0,
\end{align}
see \cite[p.121]{MerkerPorten} for a proof.  The boundedness of $\mathcal{T}$ on $\mathcal{C}^{k, \beta}(\partial \D)$  implies that there is a constant $C_{k,\beta}>1$ such that for any $v \in \mathcal{C}^{k,\beta}(\partial \D)$ we have  
\begin{align} \label{ine_chuancuaT}
\|\mathcal{T}_1 v\|_{k,\beta, \partial \D} \le C_{k,\beta} \|v\|_{k,\beta,\partial \D}.
\end{align}
Extending $u_0, \mathcal{T}_1 u_0$ harmonically to $\D.$ By construction, the function $f(z):= -\mathcal{T}_1 u_0(z) + i u_0(z)$ is holomorphic on  $\D$ and continuous on $\overline{\D}$ provided that $u_0$ is in $\mathcal{C}^{\beta}(\partial \D)$ with $0< \beta <1.$ By \cite[Th. 4.2]{MerkerPorten2}, $\|f\|_{k,\beta, \overline{\D}}$ is bounded by $ \|f\|_{k,\beta,\partial \D}$ times a constant depending only on $(k,\beta).$
Since $\| u_0\|_{k,\beta, \overline{\D}} \le \|f\|_{k,\beta, \overline{\D}}$ and $\|f\|_{k,\beta,\partial\D}\le  (1+ C_{k,\beta}) \|u_0\|_{k,\beta,\partial \D}$ by (\ref{ine_chuancuaT}), we have
\begin{align} \label{ine_danhgiachuaCkcuauvoibien}
\| u_0\|_{k,\beta,\overline{\D}} \le C'_{k,\beta} \| u_0\|_{k,\beta, \partial \D}, 
\end{align}
for  some constant $C'_{k,\beta}$ depending only on $(k,\beta).$ A direct consequence of the above inequalities is that when $u_0$ is smooth on $\partial \D$, the associated holomorphic function $f$ is also smooth on $\overline{\D}.$


\subsection{Analytic discs half-attached to $\R^n$ in $\C^n$} \label{subsec_1}
The goal of this subsection is to construct a special family of analytic discs half-attached to $\R^n$ in $\C^n.$ The main result is Proposition \ref{pro_specialanalyticdisc} presented at the end of the subsection. The reader should keep in mind that the idea that we use below will be constantly applied later.

 In what follows, we identify $\C^n$ with $\R^n+ i \R^n.$ Let $\mathbf{z} \in B^*_{2n}(0,1).$ Let $u=(u_1,\cdots,u_n)$ be a vector with components $u_j  \in C^{k, \beta}(\partial \D)$ for $1 \le  j \le n$ such that $u \equiv 0$ on $\partial^+ \D.$ Then, $\mathcal{T}_1 u:= (\mathcal{T}_1 u_1, \cdots, \mathcal{T}_1 u_n) $ is a vector in $\big( C^{k, \beta}(\partial \D)\big)^n.$  As above, extend $u$ and  $\mathcal{T}_1u$  harmonically to $\D.$
By the last subsection, $u$ and $\mathcal{T}_1 u$ belong to $\big(\mathcal{C}^{k,\beta}(\overline{\D})\big)^n.$  It follows that the map 
$$f:= -\mathcal{T}_1u+ i u$$
 is a $\mathcal{C}^{k,\beta}$ mapping from $\overline{\D}$ to $\C^n$ which is holomorphic on  $\D$ and $f|_{\partial^+ \D} \subset \R^n.$ In other words, $f$ is a $\mathcal{C}^{k,\beta}$ analytic disc half-attached to $\R^n$ in $\C^n$ with $f(1)=0.$  We are going to choose $u$ depending on the parameter $\mathbf{z}$ in a small enough ball centered at $0$ such that there exist $z^* \in \D$ and a constant $c_0$ independent of $\mathbf{z}$ for which
\begin{align} \label{condition00}
\|f\|_{3,\overline{\D}} \le c_0, 
\end{align}
and
\begin{align}\label{condition0}
f(z^*)=\mathbf{z} \quad \text{and} \quad |1- z^*|\le c_0 |\mathbf{z}|.
\end{align}
Recall that we will systematically identify continuous functions on $\partial \D$ with their harmonic extension to $\D.$ Hence, for any continuous function $u$ on $\partial \D,$ we can speak of its derivatives in $(x,y)$ as the ones of its harmonic extension, where $z=x+iy \in \D.$

\begin{lemma} \label{le_version8existenceu} There exists a function  $u \in \mathcal{C}^{\infty}(\partial \D)$ vanishing on $\partial^+ \D$ so that $\partial_x u(1)=-1.$
\end{lemma}
\begin{proof} Differentiating (\ref{equ_harmonicextension}) gives 
$$\partial_x u(1)=\frac{1}{2\pi} \int_{-\pi}^{\pi} \frac{u(e^{i\theta})}{\cos\theta -1}d\theta.$$
Note that the last integral is well-defined because $u$ vanishes on $\partial^+ \D.$ It is easy to choose a smooth $u$ so that the above integral is equal to $-1$ and $u \equiv 0$ on $\partial^+ \D.$ The proof is finished.   
\end{proof}

\begin{lemma} \label{le_taylortai1aacuau} Let $u$ be a functions as in Lemma \ref{le_version8existenceu}.  Then there exist two smooth functions $g_1, g_2$ defined on $[0,1]$ so that $u(1-s+i s)= s+ s^2 g_1(s)$ and $-\mathcal{T}_1u(1-s+ i s)=  s+ s^2g_2(s),$ for every $s \in [0,1].$
\end{lemma}
\begin{proof} 
On the other hand, the Cauchy-Riemann equations imply $-\partial_x \mathcal{T}_1u(1)= \partial_y u(1)=0$ (because $u$ vanishes on $\partial^+ \D$) and $-\partial_y \mathcal{T}_1u(1)= -\partial_x u(1)= 1.$  Define 
$$g_1(s):= \frac{1}{2}\int_0^1 (\partial^2_x- 2 \partial_x\partial_y+ \partial^2_y)u\big( t(1-s+i s)+ (1-t)\big) dt$$
and  
$$\quad g_2(s):= - \frac{1}{2}\int_0^1 (\partial^2_x- 2 \partial_x\partial_y+ \partial^2_y) \mathcal{T}_1 u\big( t(1-s+i s)+ (1-t)\big) dt.$$  
A direct application of Taylor's expansions to $u$ and $-\mathcal{T}_1u$ shows that $g_1$ and $g_2$ satisfy the desired property.
\end{proof}

We will repeatedly use the following known version of the inverse function theorem. 

\begin{lemma}\label{le_implicitfunction2} Let $m\in \N^*.$ Let $\Phi_0$ be a function from $B_m(0,1)$ to $\R^m.$ Assume that there are a  nondegenerate square matrix $A$ of order $m$ and a $M$-Lipschitz function $g$ on $B_m(0,1)$ for some constant $M>0$ such that 
$$\Phi_0(\mathbf{z})= A \mathbf{z}+ g(\mathbf{z}),$$
 for every $\mathbf{z} \in B_m(0,1)$ and $g(0)=0,$ $|A^{-1}|_{\ln} M <1,$ where $|A^{-1}|_{\ln}$ is the norm of the linear self-map of $\R^m$ associated with $A^{-1}.$ Then, for every $0<r<1$ and every $\tilde{\mathbf{z}} \in B_m\big(0, \frac{1-|A^{-1}|_{\ln} M }{|A^{-1}|_{\ln}} r \big),$ there exists a unique point $\mathbf{z}^* \in B_m(0,r)$ such that $\Phi_0(\mathbf{z}^*)= \tilde{\mathbf{z}}.$ 
\end{lemma}
\begin{proof} Since $g$ is $M$-Lipschitz on $B_m(0,1),$ we have 
$$|g(\mathbf{z})- g(\mathbf{z}')| \le M |\mathbf{z}- \mathbf{z}'|,$$  
for all $\mathbf{z}, \mathbf{z}' \in B_m(0,1).$ In particular, we have $|g(\mathbf{z})| \le M |\mathbf{z}|$ because $g(0)=0.$ Let $\tilde{\mathbf{z}}$ be a point in $B_m\big(0, \frac{1-|A^{-1}|_{\ln}M }{|A^{-1}|_{\ln}} r \big).$ The equation $\Phi_0(\mathbf{z})=\tilde{\mathbf{z}}$ is equivalent to $\mathbf{z}= A^{-1}\big(\tilde{\mathbf{z}} - g(\mathbf{z})\big).$ Let $r \in (0,1).$ Define 
$$R(\mathbf{z}):= A^{-1}\big(\tilde{\mathbf{z}} - g(\mathbf{z})\big),$$
 for $\mathbf{z} \in \overline{B}_m(0,r).$ Then $R$ is a self-map of  $\overline{B}_m(0,r).$ Indeed, we have 
 $$|R(\mathbf{z})| \le |A^{-1}|_{\ln} |\tilde{\mathbf{z}} - g(\mathbf{z})| \le  |A^{-1}|_{\ln} \big(|\tilde{\mathbf{z}}|+ M |\mathbf{z}|\big) \le |A^{-1}|_{\ln}\big(  \frac{1-|A^{-1}|_{\ln} M }{|A^{-1}|_{\ln}} r+ M r\big) = r,$$
 for any $\mathbf{z} \in \overline{B}_m(0,r).$ Additionally, similar estimates also gives
$$|R(\mathbf{z})- R(\mathbf{z}')| \le |A^{-1}|_{\ln}M |\mathbf{z}- \mathbf{z}'|.$$  
Since $|A^{-1}|_{\ln} M<1,$ $R$ is a contraction of $\overline{B}_m(0,r).$ Since the last metric space is compact, the fixed point theorem applied to $R$ implies that $R$ has a unique fixed point $\mathbf{z}^* \in \overline{B}_m(0,r).$  Equivalently, there is a unique point $\mathbf{z}^* \in \overline{B}_m(0,r)$ for which $\Phi_0(\mathbf{z}^*)= \tilde{\mathbf{z}}.$ 
\end{proof}

For any two vectors $\mathbf{v}^j=(v^j_1, \cdots, v^j_n) \in \R^n$ with $j=1$ or $2,$  we denote by $\mathbf{v^1}\cdot \mathbf{v}^2$ the vector in $\R^n$ whose $l^{th}$ component is $v^1_l v^2_l$ for $1\le l\le n.$  Let $u$ be a function as in Lemma \ref{le_version8existenceu}. By abuse of notation, denote also by $u$ the vector of $\mathcal{C}^{\infty}(\partial\D)^n$ whose components are all equal to $u.$ We define 
\begin{align} \label{defcuau_version8}
u_{\mathbf{z},t}(e^{i\theta}):= t  u(e^{i\theta}) \cdot \frac{\Im \mathbf{z}}{|\mathbf{z}|}, 
\end{align}
for any $\mathbf{z} \in B^*_{2n}(0,1)$ and $t \in (0,1].$ Extend $u_{\mathbf{z},t}$ harmonically to $\D.$ Define 
\begin{align} \label{defcuaF_version8}
F(z, \mathbf{z},t):=  t (\Re \mathbf{z}- \Im \mathbf{z})  -\mathcal{T}_1(u_{\mathbf{z},t})(z)+ iu_{\mathbf{z},t}(e^{i\theta}),
\end{align}
for any $z \in \overline{\D},$ $\mathbf{z} \in B^*_{2n}(0,1)$ and $t \in (0,1].$ We have following properties of $F.$ 

\begin{proposition}\label{pro_specialanalyticdisc} The map  $F: \overline{\D} \times B^*_{2n}(0,1) \times (0,1] \rightarrow  \C^n$ is smooth and the three following conditions hold:

$(i)$ for any $\mathbf{z} \in B^*_{2n}(0,1)$ and $t \in (0,1],$ the mapping $F(\cdot, \mathbf{z},t)$ is a smooth analytic disc half-attached to $\R^n$ in $\C^n,$ and $F(1,\mathbf{z},t)=  t (\Re \mathbf{z}- \Im \mathbf{z})  \in B_n(0,2t) \subset \R^n,$

$(ii)$ there exists a constant $r_0>0$ so that for any $\mathbf{z} \in B^*_{2n}(0,r_0)$ and $t \in (0,1],$  there exists $\mathbf{z}^* \in B^*_{2n}(0,1)$ for which 
$$F(1-|\mathbf{z}^*|+ i |\mathbf{z}^*|, \mathbf{z}^*,t)= t \mathbf{z}$$
 and $|\mathbf{z}^*| \le 2| \mathbf{z}|,$ 

$(iii)$ there exists a constant $c_0>1$ so that  for any $\mathbf{z} \in B^*_{2n}(0,1)$ and $t \in (0,1],$ we have
\begin{align} \label{ine_danhgiadaohamchoF} 
 \|F(\cdot, \mathbf{z},t)\|_{3} \le t c_0  \, \text{ and} \quad  \|D_{\mathbf{z}} F(\cdot, \mathbf{z},t) \|_{2} \le  t c_0 |\mathbf{z}|^{-1},
 \end{align} 
 where $D_{\mathbf{z}}$ is the differential with respect to $\mathbf{z}.$
\end{proposition}
\begin{proof} The properties $(i)$ and $(iii)$ automatically hold by our construction. It remains to prove $(ii).$ Fix $t \in (0,1].$  For every $\mathbf{z} \in B^*_{2n}(0,1),$ define $\Phi(\mathbf{z}): =F(1-|\mathbf{z}|+ i |\mathbf{z}|, \mathbf{z},t)$ and $\Phi(0):=0.$ Applying Lemma \ref{le_taylortai1aacuau} to each component of  $u_{\mathbf{z},t}$ and $s= |\mathbf{z}|,$ using (\ref{defcuau_version8}) and (\ref{defcuaF_version8}), there exists a smooth map $g_0: [0,1] \rightarrow \R^n$ such that  
$$\Phi(\mathbf{z})=  t (\Re \mathbf{z}- \Im \mathbf{z}) +  t  \Im \mathbf{z}+    t|\mathbf{z}|^2 g_0(|\mathbf{z}|) \cdot \frac{\Im T(\mathbf{z})}{|\mathbf{z}|} =t \mathbf{z}+t|\mathbf{z}| g_0(|\mathbf{z}|) \cdot  \Im T(\mathbf{z}) .$$
Put 
$$g(\mathbf{z}):= t |\mathbf{z}|  g_0(|\mathbf{z}|) \cdot \Im T(\mathbf{z}).$$
 Let $r_0< 1/16  \min\{\|g_0\|^{-1}_{1}, 1\}.$  Observe that  $g(0)=0$ and
$$\|g\|_{1, B^*_{2n}(0,2r_0)} \le t/4 < t/2$$
Thus, $g$ is $t/2$-Lipschitz on $\overline{B}_{2n}(0, 2 r_0).$  Applying Lemma \ref{le_implicitfunction2} to $\Phi$ in place of $\Phi_0,$ $A= t \Id$ and $g$ as above shows that for any $\mathbf{z} \in B^*_{2n}(0,r_0),$ there exists $\mathbf{z}^* \in B^*_{2n}(0,2 r_0)$ for which $\Phi(\mathbf{z}^*)= t \mathbf{z}.$ Moreover, the last equation implies that 
$$t |\mathbf{z}| \ge t |\mathbf{z}^*| - |g(\mathbf{z}^*)| \ge  t |\mathbf{z}^*| - t/2 |\mathbf{z}^*|.$$
Hence, $|\mathbf{z}^*| \le 2 |\mathbf{z}|.$ The proof is finished.
\end{proof}
For each $\mathbf{z} \in B^*_{2n}(0,r_0),$ define $f(z):=F(z, \mathbf{z}^*,1)$ and $z^*:=1-|\mathbf{z}^*|+ i |\mathbf{z}^*|.$ It is clear that $f$ and $z^*$ satisfy the two conditions (\ref{condition00}) and (\ref{condition0}).

\subsection{Analytic discs partly attached to $(\R^+)^n$ in $\C^n$} \label{subsec_R+}
The goal of this subsection is to construct a family $F'$ of analytic discs which somewhat resembles the one in Proposition \ref{pro_specialanalyticdisc} and partly attached to $(\R^+)^n$ in $\C^n,$ where $\R^+$ is the set of nonnegative real numbers. The arguments used in the last subsection do not permit us to control the position of the part of the boundary of the disc in $\R^n.$ The idea is to construct  discs which look like the image of $F$ under the map $(z_1,\cdots,z_n) \longmapsto (z_1^2,\cdots,z_n^2),$ this image is half-attached to $(\R^+)^n,$ where $F$ is the family in the last subsection.

At the end of this subsection, we also introduce an another family $F'_{\boldsymbol{\tau}}$ of discs half-attached to $\R^n$ parametrized by $\boldsymbol{\tau} \in B_n(0,2)$ which contains $F'$ as a subfamily. Let us explain why we need such $F'_{\boldsymbol{\tau}}.$ In the general case considered in Section \ref{secgeneral}, the required analytic discs in Proposition \ref{pro_familydiscK} can be obtained as a small perturbation of $F'.$ Due to the nonsmoothness of $(\R^+)^n$ (or the submanifold $K$ with singularity in the general case), any family of discs partly attached to $(\R^+)^n$ is generally no longer so when being perturbed. Hence, in order to control the perturbed family, one should embed $F'$ in the bigger family $F'_{\boldsymbol{\tau}}$ which is more stable under perturbation.

Define 
\begin{align} \label{def_rho12}
\rho_1(\theta):= \frac{1}{2\pi(\cos\theta -1)} \quad \text{and} \quad \rho_2(\theta):=- \frac{\sin \theta}{2\pi(\cos\theta-1)^2} , 
\end{align}
for $\theta \in [-\pi, \pi].$ 

\begin{lemma} \label{le_tinhcacdaohamcuau} Let $u$ be harmonic function on $\D$ and smooth on $\overline{\D}.$ Assume that $u$ vanishes on $\partial^+ \D.$ Then, we have $\partial_y u(1)=0,$ $\partial^2_y u(1)=\partial_x u(1),$ $\partial^2_x u(1)= - \partial_x u(1)$  and  
$$\partial_x u(1)= \int_{-\pi}^{\pi} u(e^{i\theta}) \rho_1(\theta) d\theta \quad \text{and} \quad \partial_x \partial_y u(1)=\int_{-\pi}^{\pi} u(e^{i\theta}) \rho_2(\theta) d\theta.$$
\end{lemma}
\begin{proof}
Firstly, observe that for an arbitrary $\mathcal{C}^2$ function $u(x+ i y)$ on $\overline{\D},$ we have 
\begin{align} \label{eq_daohamvhailantheotheta''}
\partial_{\theta} u(e^{i\theta})|_{\theta=0}=\partial_y u(1)\quad  \text{and} \quad \partial^2_{\theta} u(e^{i\theta})|_{\theta=0}= \partial_y^2 u(1)- \partial_x u(1).
\end{align}
Now let $u$ be the function in the statement.  The last two equalities combined with the fact that $u|_{\partial^+ \D}\equiv 0$ imply that 
\begin{align} \label{eq_daohamcap12tai1cuau}
\partial_y u(1)=0,\quad \partial_y^2 u(1)= \partial_{\theta}^2 u(1)+ \partial_x u(1)=\partial_x u(1).
\end{align}
Since $\Delta u(z)=0,$  we get $\partial^2_x u(1)= - \partial^2_y u(1)= -\partial_x u(1).$ On the other hand, it is computed in the proof of Lemma \ref{le_version8existenceu}  that $\partial_x u(1)= \int_{-\pi}^{\pi} u(e^{i\theta}) \rho_1(\theta) d\theta.$ 
Differentiating the Poisson kernel at $(1,0)$ gives 
\begin{align*}
\partial_x|_{x=1} \partial_y|_{y=0} \frac{1- |z|^2}{2\pi |e^{i\theta}- z|^2}&= 2\partial_x|_{x=1}  \frac{-y |e^{i\theta}-z|^2+(1-x^2-y^2)\sin \theta- y(1- |\mathbf{z}|^2) }{2\pi |e^{i\theta}- (x+i y)|^4}\bigg|_{y=0}\\
&=2 \sin\theta \, \partial_x|_{x=1} \frac{1- x^2}{2\pi(x^2- 2 x \cos\theta+1)^2}\\
&= 4 \sin \theta \, \frac{-x}{2\pi(x^2- 2 x \cos\theta+1)^2}\bigg|_{x=1} = - \frac{\sin \theta}{2\pi(\cos\theta-1)^2}= \rho_2(\theta).
\end{align*} 
Combining this with (\ref{equ_harmonicextension}) shows that $\partial_x \partial_y u(1)=\int_{-\pi}^{\pi} u(e^{i\theta}) \rho_2(\theta) d\theta.$ The proof is finished.
\end{proof}

\begin{corollary} \label{cor_khaitrientaylorcuaftai1} Let $u$ be a function as in Lemma \ref{le_tinhcacdaohamcuau}.  There are smooth functions $g_1,$ $g_2$ defined on $[0,1]$ and $g_3$ defined on $[-\pi/2,\pi/2]$ so that for any $s \in [0,1],$ we have  
$$u(1- s)=- s \partial_x u(1) - s^2 \frac{\partial_x u(1)}{2} + s^3 g_1(s) \quad  \text{and} \quad -\mathcal{T}_1 u(1-s)=  s^2 \frac{\partial_x \partial_y u(1)}{2}+ s^3 g_2(s),$$
and 
$$-\mathcal{T}_1 u(e^{i\theta})=-\partial_x u(1) \theta- \frac{\partial_x \partial_y u(1)}{2} \theta^2+ \theta^3 g_3(\theta),$$
for any $\theta \in [-\pi/2, \pi/2].$ Moreover, there is a constant $c$ independent of $u$ for which $\|g_j\|_0 \le c \|u\|_{4,\partial \D}$ for $j=1,2,3.$
\end{corollary}
\begin{proof} This is an analogue of  Lemma \ref{le_taylortai1aacuau}. Recall that we have 
$$- \partial_x \mathcal{T}_1 u(1)= \partial_y u(1)=0 \quad  \text{and} \quad - \partial_x \mathcal{T}_1 u(z)= \partial_y u(z).$$
Thus, $-\partial^2_x \mathcal{T}_1 u(z)= \partial_x \partial_y u(z).$ Letting $z=1$ in the last equality gives $-\partial^2_x \mathcal{T}_1 u(1)= \partial_x \partial_y u(1).$ Using the last equalities and (\ref{eq_daohamcap12tai1cuau}) and Taylor's expansions at $s=0$ for $u(1-s)$ and for $-\mathcal{T}_1u(1-s),$ we get $g_1,g_2$ and the first two equalities. By (\ref{eq_daohamvhailantheotheta''}), we have 
$$-\partial_{\theta}\mathcal{T}_1 u(e^{i\theta})=- \partial_y \mathcal{T}_1 u(1)=- \partial_x u(1)$$
and 
$$-\partial^2_{\theta}\mathcal{T}_1 u(e^{i\theta})=-\partial^2_y\mathcal{T}_1 u(1) - \partial_x \mathcal{T}_1 u(1)=-\partial^2_y\mathcal{T}_1 u(1)= - \partial_x \partial_y u(1).$$ 
This combined with Taylor's expansion at $\theta=0$ of $-\mathcal{T}_1 u(e^{i\theta})$ gives $g_3$ and the third equality. Since $g_j$ are the remainder in Taylor's expansions up to the order $2,$ we also see that there is a constant $c$ independent of $u$ so that for $1\le j \le 3,$ 
$$\|g_j\|_0 \le \max\{ \|u\|_{3,\overline{\D}}, \|\mathcal{T}_1 u\|_{3,\overline{\D}}\} \le c \|u\|_{4,\partial \D},$$
by (\ref{ine_danhgiachuaCkcuauvoibien}) and (\ref{ine_chuancuaT}).  The proof is finished.
\end{proof}

Let $\tilde{z} \in \D,$ $\delta \in (0,1]$ and $\gamma \in (0,1].$ We want to construct a function $u$ on $\partial \D$ which is differentiable enough such that $\partial_x u(1)$ and $\partial_x \partial_y u(1)$ equal to prescribed values. Note that we always identify $u$ with its harmonic extension on $\D.$ Precisely, we want to choose $u$ so that 
\begin{align} \label{eq_choiceofdaohamu}
 - \delta \partial_x u(1) -\delta^2 \frac{\partial_x u(1)}{2}= \Im \tilde{z} \quad  \text{and} \quad   \delta^2 \frac{\partial_x \partial_y u(1)}{2}= \Re \tilde{z}- \gamma.
\end{align}
The last system is equivalent to 
 \begin{align} \label{eq_choice2ofdaohamu}
\partial_x u(1)= -\frac{2\Im \tilde{z}}{\delta(2+\delta)} \quad  \text{and} \quad  \partial_x \partial_y u(1)= -\frac{2(\gamma- \Re \tilde{z})}{\delta^2} \cdot
\end{align}
In order to construct a such $u$ satisfying the last property, we will need the following lemma.

 \begin{lemma} \label{le_chonmayhamtuyentinhdoclap} Let $m$ be a positive integer. Let $\{a_j\}_{1\le j\le m}$ be real smooth functions on $\partial \D$ such that they are linearly independent in $\mathcal{C}^{\infty}(\partial \D).$ Then there exist $b_j \in \mathcal{C}^{\infty}(\partial \D)$ with $1\le j \le m$ so that 
$$\int_{-\pi}^{\pi} b_j(e^{i\theta}) a_{j'}(e^{i\theta}) d\theta= \delta^j_{j'},$$
 for all $1\le j,j' \le m,$ where $\delta^j_{j'}$ is the Kronecker delta.   
 \end{lemma}
 \begin{proof} Let $L_j: \mathcal{C}^{\infty}(\partial \D) \rightarrow \R$ be the linear functional defined by 
$$L_j(v)= \int_{-\pi}^{\pi} a_j(e^{i\theta}) v(e^{i\theta}) d \theta,$$ 
for $v \in \mathcal{C}^{\infty}(\partial \D)$ and $1\le j \le m.$  The linear independence of $a_j$ and the density of smooth functions in $L^2(\partial \D)$ imply that $\{L_j\}_{1\le j \le m}$ are linearly independent. A basic result of linear algebra says that for any $1\le j \le m,$ 
$$\big(\bigcap_{j' \not =j} \Ker L_{j'} \big) \backslash \Ker L_{j} \not = \varnothing.$$
In the other words, there is  $b_j \in \mathcal{C}^{\infty}(\partial \D)$ satisfying $L_{j'}(b_j)= \delta^j_{j'}.$ The proof is finished.
 \end{proof}

\begin{lemma} \label{le_dinhnghiau0u1}
There exist  two functions $u_1(e^{i\theta}), u_2(e^{i\theta}) \in \mathcal{C}^{\infty}(\D)$ vanishing on $\partial^+\D $ such that 
\begin{align} \label{eq_u0u1daoham}
\partial_x u_1(1)=\partial_x\partial_y u_2(1)=1 \quad  \text{and} \quad \partial_x u_2(1)=\partial_x\partial_y u_1(1)=0.
\end{align}
\end{lemma}
\begin{proof} By  Lemma \ref{le_tinhcacdaohamcuau}, the condition (\ref{eq_u0u1daoham}) is equivalent to
 $$\int_{-\pi}^{\pi} u_1 \rho_1 d\theta= \int_{-\pi}^{\pi} u_2 \rho_2 d\theta=1 \quad \text{and} \quad \int_{-\pi}^{\pi} u_1 \rho_2 d\theta= \int_{-\pi}^{\pi} u_2 \rho_1 d\theta=0.$$
 Put $\partial^- \D= \partial \D \backslash \partial^+ \D.$ Let $\chi \in \mathcal{C}^{\infty}(\partial \D)$ with $\supp \chi \subset \partial^- \D$ and $\chi \not \equiv 0.$  
Let  $a_1=\chi \rho_1(\theta), \,a_2= \chi \rho_2(\theta).$ Observe that these functions are linearly independent in $\mathcal{C}^{\infty}(\partial\D).$ This allows us to apply Lemma \ref{le_chonmayhamtuyentinhdoclap} to $a_1, a_2.$ Hence,  we obtain $b_1, b_2 \in \mathcal{C}^{\infty}(\partial \D)$ with 
$$\int_{-\pi}^{\pi} b_j a_{j'} d \theta= \delta^{j}_{j'}.$$
Let $u_1:= \chi b_1$ and $u_2:= \chi b_2.$  One easily checks that $u_1$ and $u_2$ satisfies the desired property. The proof is finished.
\end{proof}
Define 
\begin{align}\label{choice of u}
u_{\tilde{z},\delta, \gamma}(e^{i\theta})=-u_1(e^{i\theta}) \frac{2\Im \tilde{z}}{\delta(2+\delta)}-  u_2(e^{i\theta}) \frac{2(\gamma- \Re \tilde{z})}{\delta^2} \cdot
\end{align}
We deduce from Lemma \ref{le_dinhnghiau0u1} and (\ref{eq_choice2ofdaohamu}) that $u_{\tilde{z},\delta, \gamma}$ enjoys the property (\ref{eq_choiceofdaohamu}). The following explains our choice of $u_{\tilde{z},\delta, \gamma}.$

\begin{lemma} \label{le_analyticdiscuz0deltagamma} Let $(\tilde{z},\delta,\gamma) \in \D \times (0,1]^2$ so that 
\begin{align} \label{ine_conditionsondeltagamma}
\gamma \ge  2 |\tilde{z}| \quad \text{and} \quad  2 \sqrt{\gamma} \ge \delta \ge \frac{\sqrt{\gamma}}{2}.
\end{align}
Then, there are a positive constant $\theta_c$ independent of $(\tilde{z},\delta,\gamma)$ and a smooth function $g_{\tilde{z}, \delta, \gamma}(s)$ defined on $[0,1]$ depending smoothly on the parameter $(\tilde{z},\delta,\gamma)$ such that $\|g_{\tilde{z}, \delta, \gamma}\|_1$ is bounded independently of $(\tilde{z},\delta,\gamma)$  and the analytic disc 
$$f_{\tilde{z}, \delta, \gamma}:=\gamma -\mathcal{T}_1 u_{\tilde{z},\delta, \gamma}+ i u_{\tilde{z},\delta, \gamma}$$
 is $[e^{-i \theta_0}, e^{i\theta_0}]$-attached to $\R^+$ in $\C$ and $f_{\tilde{z}, \delta, \gamma}(1-\delta)= \tilde{z} + \delta^3 g_{\tilde{z}, \delta, \gamma}(\delta).$ Moreover, the quantities
$$\delta \|D_{\delta}\, g_{\delta}(\cdot)\|_0 \quad \text{and} \quad \|D_{(\tilde{z},\gamma)}g_{\tilde{z}, \delta, \gamma}(\cdot)\|_0$$
are bounded independently of $(\tilde{z},\delta,\gamma),$ where $g$ is considered as a function of  $(s,\tilde{z},\delta,\gamma)$ and $D_{(\tilde{z},\delta,\gamma)}$ is the differential with respect to $(\tilde{z},\delta,\gamma).$
\end{lemma}
\begin{proof}  Corollary (\ref{cor_khaitrientaylorcuaftai1}), (\ref{choice of u}) and (\ref{eq_choiceofdaohamu}) show that there exist  smooth functions $g_1,g_2$ defined on $[0,1]$ depending smoothly on  $(\tilde{z},\delta,\gamma)$  for which 
$$f(1- \delta)= \gamma + \Re \tilde{z} - \gamma+ \delta^3 g_1(\delta) + i\big(\Im \tilde{z}+ \delta^3 g_2(\delta)\big)= \tilde{z}+ \delta^3\big(g_1(\delta)+ i g_2(\delta)).$$
Hence, it is immediate to see that the function
$$g_{\tilde{z}, \delta, \gamma}(\delta):=g_1(\delta)+ i g_2(\delta)$$
 satisfies $f(1-\delta)= \tilde{z} + \delta^3 g(\delta).$ By the hypothesis on $(\tilde{z},\delta,\gamma),$ we have 
\begin{align}\label{ine_deltaz0gammadieukien}
\big|\frac{2\Im \tilde{z}}{\delta(2+\delta)}\big| \le 2, \quad   \frac{1}{4} \le \frac{2(\gamma- \Re \tilde{z})}{\delta^2}\le  12.
\end{align}
This yields that  $\|g_1\|_1$ and $\|g_2\|_1$ are bounded  independently of $(\tilde{z},\delta,\gamma),$ hence, so is $\|g\|_1.$ Estimating $\delta\|D_{\delta}\, g_{\tilde{z}, \delta, \gamma}(\cdot)\|_0$ and $\|D_{(\tilde{z},\gamma)}g_{\tilde{z}, \delta, \gamma}(\cdot)\|_0$ is done similarly.  

Now we prove that $f_{\tilde{z}, \delta, \gamma}$ is partly attached to $\R^+.$ To this end, it suffices to check the sign of the real part of  $f_{\tilde{z}, \delta, \gamma}.$ Using again Corollary (\ref{cor_khaitrientaylorcuaftai1}), (\ref{choice of u}) and (\ref{eq_choiceofdaohamu}) implies that for $\theta \in [-\pi/2, \pi/2],$ we have
$$\Re f_{\tilde{z}, \delta, \gamma}(e^{i\theta})=\gamma+ \frac{2\Im \tilde{z}}{\delta(2+\delta)}\theta + \frac{2(\gamma- \Re \tilde{z})}{\delta^2} \theta^2 + \theta^3 g_3(\theta),$$
where $g_3(\theta)$ is a smooth function on $[-\pi/2,\pi/2]$ whose supnorm is bounded  independently of $(\tilde{z},\delta,\gamma).$ Let $c$ be a such upper bound of $\|g_3\|_0.$   Put 
$$\tilde{f}_1(\theta):=\gamma+ \frac{2\Im \tilde{z}}{\delta(2+\delta)}\theta + \frac{\gamma- \Re \tilde{z}}{\delta^2} \theta^2$$
and 
$$\tilde{f}_2(\theta):= \frac{\gamma- \Re \tilde{z}}{\delta^2} \theta^2 + \theta^3 g_3(\theta).$$
We have $\Re f_{\tilde{z}, \delta, \gamma}(e^{i\theta})= \tilde{f}_1(\theta)+ \tilde{f}_2(\theta).$ By the second inequality of
(\ref{ine_deltaz0gammadieukien}), one sees that 
$$\tilde{f}_2(\theta) \ge \theta^2/8 + \theta^3 g_3(\theta) \ge \theta^2\big(1/8 - |\theta| \|g_3\|_0) \ge 0$$
provided that $|\theta| \le \min\{\pi/2, 1/(8c) \}.$  Observe that $\tilde{f}_1(\theta)$ is a quadratic polynomial in $\theta.$ Its discriminant equals to 
$$\frac{1}{\delta^2}\bigg[ \frac{\Im^2 \tilde{z}}{(2+ \delta)^2} - \gamma (\gamma - \Re \tilde{z}) \bigg] \le \frac{1}{\delta^2}\big[\Im^2 \tilde{z} - \gamma^2/2 \big] \le 0,$$
because $\delta \ge 0$ and $\gamma \ge 2 |\tilde{z}|.$ This means that $\tilde{f}_1(\theta)\ge 0$ for all $\theta.$ Hence, $\Re f_{\tilde{z}, \delta, \gamma}(e^{i\theta}) \ge 0$ for  
$|\theta| \le \theta_0:=\min\{\pi/2, 1/(8c) \}.$ The proof is finished.
\end{proof}

 Now let $\mathbf{z}=(z_1,\cdots,z_n) \in B^*_{2n}(0,\frac{1}{2n})$ and let $t \in (0,1].$ In the formula (\ref{choice of u}), let 
$$\gamma= 2 |\mathbf{z}|,\quad \delta= \sqrt{|\mathbf{z}|} \quad  \text{and} \quad \tilde{z}=z_j,$$
where $z_j$ is the $j^{th}$ component of $\mathbf{z},$ denote by $u'_{\mathbf{z}; j}$ the function $u_{\tilde{z}, \delta, \gamma}$ with the above choice of $(\tilde{z},\delta,\gamma).$

Define  $u'_{\mathbf{z},t}$ to be the vector of $\mathcal{C}^{\infty}(\D)^n$ whose $j^{th}$ component is equal $t u'_{\mathbf{z}; j},$ for $1\le j \le n.$ Extend $u'_{\mathbf{z},t}$ harmonically to $\D.$ Define 
\begin{align} \label{def_dinhghiaF'}
F'(z, \mathbf{z},t):= 2t (|\mathbf{z}|, \cdots, |\mathbf{z}|) -\mathcal{T}_1 (u'_{\mathbf{z},t})(z) + i u'_{\mathbf{z},t},
\end{align}
for $z \in \overline{\D},$ where  $(|\mathbf{z}|, \cdots, |\mathbf{z}|) \in \R^n.$ Then, $F'$ is a family of analytic discs half-attached to $\R^n$ and $F'(1,\mathbf{z},t)= t (|\mathbf{z}|, \cdots, |\mathbf{z}|) \in (\R^+)^n.$ The following is just a reformulation of  (\ref{eq_choice2ofdaohamu}).
\begin{align} \label{eq_daohamtheoxycuau'zt}
 \partial_x  u'_{\mathbf{z},t}(1)= -\frac{2t\Im \mathbf{z}}{\sqrt{|\mathbf{z}|}(2+ \sqrt{|\mathbf{z}|})} \quad  \text{and} \quad  \partial_x \partial_y u'_{\mathbf{z},t}(1)= -\frac{2t(2|\mathbf{z}|- \Re \mathbf{z})}{|\mathbf{z}|},
\end{align} 
where we wrote $|\mathbf{z}|$ for $(|\mathbf{z}|,\cdots, |\mathbf{z}|)$ in the last equality.

\begin{proposition}\label{pro_specialanalyticdisc'}  The map  $F': \overline{\D} \times B^*_{2n}(0,\frac{1}{2n}) \times (0,1] \rightarrow  \C^n$ is smooth and there exists two constants $r'_0 \in (0,1/4)$ and $c_0 >1$ such that the three following conditions hold:

$(i)$ for any $\mathbf{z} \in B^*_{2n}(0, \frac{1}{2n})$ and $t \in (0,1],$ the mapping $F'(\cdot, \mathbf{z},t)$ is a smooth analytic disc $[e^{ic_0^{-1}}, e^{ic_0^{-1}}]$-attached to $(\R^+)^n$ in $\C^n,$ and 
$$F'(1,\mathbf{z},t)=  2t (|\mathbf{z}|, \cdots, |\mathbf{z}|)  \in B_n(0,1) \cap (\R^+)^n,$$

$(ii)$ for any $\mathbf{z} \in B^*_{2n}(0,r'_0)$ and $t \in (0,1],$  there exists an $\mathbf{z}^* \in B^*_{2n}(0, 2 r'_0)$ for which 
$$F'(1-|\mathbf{z}^*|, \mathbf{z}^*,t)= t \mathbf{z}$$
and $|\mathbf{z}^*| \le 2| \mathbf{z}|,$ 

$(iii)$ for any $\mathbf{z} \in B^*_{2n}(0, \frac{1}{2n})$ and $t \in (0,1],$ we have
\begin{align} \label{ine_danhgiadaohamchoF'} 
 \|F'(\cdot, \mathbf{z},t)\|_{5} \le t c_0  \, \text{ and} \quad  \|D_{\mathbf{z}} F'(\cdot, \mathbf{z},t) \|_{4} \le  t c_0 |\mathbf{z}|^{-1},
 \end{align} 
 where $D_{\mathbf{z}}$ is the differential with respect to $\mathbf{z}.$
\end{proposition}

\begin{proof}  Since $F'(z,\mathbf{z},t)= t F'(z,\mathbf{z},1),$ it is enough to verify the three above conditions for $t=1.$ It is clear that  $(\gamma, \delta, \tilde{z})=(2|\mathbf{z}|, \sqrt{|\mathbf{z}|}, z_j)$ satisfies the condition (\ref{ine_conditionsondeltagamma}) for $1 \le j \le n$.  Hence, direct consequences of Lemma \ref{le_analyticdiscuz0deltagamma} and (\ref{def_dinhghiaF'}) show that there exists a constant $c_0>1$ for which  the property $(i)$ and $(iii)$ hold. It remains to verify $(ii).$ We will use the same idea as in the proof of Proposition \ref{pro_specialanalyticdisc}.

Fix $t \in (0, 1].$ Let $\Phi'(\mathbf{z}): =F'(1-|\mathbf{z}|, \mathbf{z},t)$ for $\mathbf{z} \in B^*_{2n}(0,\frac{1}{2n})$ and $\Phi'(0):=0.$ By the above reason and Lemma \ref{le_analyticdiscuz0deltagamma}, there exists a smooth map $g_{\mathbf{z}}(s): [0,1] \rightarrow \R^n$ depending smoothly on $\mathbf{z} \in  B^*_{2n}(0, \frac{1}{2n})$ such that  
$$\Phi'(\mathbf{z})=t \mathbf{z}+t|\mathbf{z}|^{3/2} g_{\mathbf{z}}(\sqrt{|\mathbf{z}|}).$$
Note that the homogeneity of $F'$ in $t$ implies that $g_{\mathbf{z}}$ is independent of $t.$ Put 
$$g'(\mathbf{z}):= t |\mathbf{z}|^{3/2}  g_{\mathbf{z}}(\sqrt{|\mathbf{z}|}).$$
Observe that  $g'(0)=0$ and
\begin{align*}
D_{\mathbf{z}} g'(\mathbf{z})= \frac{3t |\mathbf{z}|^{1/2}}{2}  g_{\mathbf{z}}(|\sqrt{|\mathbf{z}|})+ t  |\mathbf{z}|^{3/2}\big\{D_{\mathbf{z}} g_{\mathbf{z}}(\sqrt{|\mathbf{z}|})+  D_s g_{\mathbf{z}}(\sqrt{|\mathbf{z}|})D_{\mathbf{z}}\sqrt{|\mathbf{z}|}\big\}.
\end{align*}
Lemma \ref{le_analyticdiscuz0deltagamma} for $(\gamma, \delta, \tilde{z})=(2|\mathbf{z}|, \sqrt{|\mathbf{z}|}, z_j)$ implies that  $\sqrt{|\mathbf{z}|}D_{\mathbf{z}} g_{\mathbf{z}}(\sqrt{|\mathbf{z}|})$ and $D_{s} g_{\mathbf{z}}(\sqrt{|\mathbf{z}|})$ are bounded independently of $\mathbf{z}.$ As a consequence, we have
$$|D_{\mathbf{z}} g'(\mathbf{z}) | \le c |\mathbf{z}|^{1/2} t,$$
for some constant $c$ independent of $(\mathbf{z},t).$ Let $r'_0:= (2c)^{-2}/3.$ The last inequality yields that $|D_{\mathbf{z}} g'(\mathbf{z}) | \le t/2$ for $\mathbf{z} \in B^*_{2n}(0,3 r'_0 ).$  Thus, $g'$ is $t/2$-Lipschitz on $\overline{B}_{2n}(0, 2 r'_0).$  Applying Lemma \ref{le_implicitfunction2} to $\Phi'$ in place of $\Phi_0,$ $A= t \Id$ and $g'$ as above shows that for any $\mathbf{z} \in B^*_{2n}(0,r'_0),$ there exists $\mathbf{z}^* \in B^*_{2n}(0,2 r'_0)$ for which $\Phi'(\mathbf{z}^*)= t \mathbf{z}.$ Moreover, the last equation implies that 
$$t |\mathbf{z}| \ge t |\mathbf{z}^*| - |g(\mathbf{z}^*)| \ge  t |\mathbf{z}^*| - t/2 |\mathbf{z}^*|.$$
Hence, $|\mathbf{z}^*| \le 2 |\mathbf{z}|.$ The proof is finished.
\end{proof}

As explained at the beginning, let us now introduce a new parameter $\boldsym{\tau} \in B_n(0,2)$ and a family $F'_{\boldsym{\tau}}$ of analytic discs half-attached to $\R^n$ contains $F'$ as a subfamily. 

\begin{lemma} \label{le_dunghamungaphayu''} Let $u_1$ be the function in Lemma  \ref{le_dinhnghiau0u1}. Then, the function $\tilde{u}:= 10 u_1$ is smooth on $\partial \D$ and vanishes on $\partial^+ \D$ and  
\begin{align} \label{ine_daohamcuaungaphayu''}
\partial_x \tilde{u}(1)=10, \quad  \partial_{x} \partial_y \tilde{u}(1)=0.
\end{align}
\end{lemma}
\begin{proof} This is obvious by the properties of $u_1.$ The proof is finished.
\end{proof}
Let $\mathbf{z}$ and $t$ be as above.  Put 
\begin{align} \label{de_dinhnghiacuau'}
u'_{\mathbf{z},t, \boldsymbol{\tau}}(z):=u'_{\mathbf{z},t}(z)+ t \boldsymbol{\tau} \cdot \tilde{u}(z),
\end{align}
where $\boldsymbol{\tau} \in B_n(0,2).$ The parameter  $\boldsymbol{\tau}$ will play a role as a control parameter.  Define 
\begin{align} \label{de_dinhnghiacuaF'}
F'_{\boldsymbol{\tau}}(z, \mathbf{z},t):=2t (|\mathbf{z}|, \cdots, |\mathbf{z}|) - \mathcal{T}_1 u'_{\mathbf{z},t,\boldsymbol{\tau}}(z)+ i u'_{\mathbf{z},t,\boldsymbol{\tau}}(z),
\end{align}
for $\mathbf{z} \in B^*_{2n}(0, \frac{1}{2n}), t\in (0,1]$ and $\boldsymbol{\tau} \in B_n(0,2).$  By construction, $F'_{\boldsymbol{\tau}}$ is a family of discs half-attached to $\R^n$ and when $\boldsymbol{\tau}=0$ we have $F'_{0} \equiv F'$ which is the family constructed earlier. By choosing the constant $c_0$ in Proposition \ref{pro_specialanalyticdisc'}  big enough, for any $\mathbf{z} \in B^*_{2n}(0, \frac{1}{2n}), t\in (0,1]$ and $\boldsymbol{\tau} \in B_n(0,2),$ we have
\begin{align} \label{ine_danhgiadaohamchoF'tau} 
 \|D^j_{\boldsym{\tau}}F'_{\boldsym{\tau}}(\cdot, \mathbf{z},t)\|_{4} \le t c_0  \, \text{ and} \quad  \|D^j_{\boldsym{\tau}}D_{\mathbf{z}} F'_{\boldsym{\tau}}(\cdot, \mathbf{z},t) \|_{3} \le  t c_0 |\mathbf{z}|^{-1},
 \end{align} 
 for $j=0,1$  and
 \begin{align} \label{eq_daohamtheotautai0cuau'tau}
 D^2_{\boldsym{\tau}} F'_{\boldsym{\tau}}(1, \mathbf{z},t) \equiv 0, \quad D_{\boldsym{\tau}} u'_{\mathbf{z},t,\boldsym{\tau}}(1)=10t,
 \end{align}
where the right-hand side of the last equality denotes the diagonal matrix of order $n$ whose coefficients on the diagonal are all equal to $10t.$

\section{Analytic discs partly attached to $K$ } \label{secgeneral}

Fix a smooth Riemannian metric  on $X.$  Let $p_0$ be an arbitrary point of $K.$ Our goal is to construct special families of analytic discs  partly attached to $K$ in a small neighborhood of $p_0$ in $X.$  Since $K$ is a generic submanifold, its dimension is at least $n.$ We first study the case where $\dim_{\R} K=n.$ Then we deduce the case of higher dimension by considering  (local) generic submanifolds of $K$.  In what follows, the notations $\gtrsim$ and $\lesssim$ respectively mean $\ge$ and $\le$ up to a positive constant depending only on the geometry of $(K,X).$ 


\subsection{The case where $K$ has no singularity}
In this subsection, we consider the case where $K$ has no singularity and $\dim_{\R} K=n$.  The $\mathcal{C}^3$-differentiability of $K$ is enough for our proof. The local coordinates described in Lemma \ref{le_localcoordinates} below are used widely in the Cauchy-Riemann geometry. Since we need to use concrete estimates uniform in $p_0$, a complete proof will be presented. We refer to the beginning of Subsection \ref{subsec_Hilberttransform} for the notation of the norms of maps below.

\begin{lemma} \label{le_localcoordinates} There exist  constants $c_1, r_K>1$ depending only on $(K,X)$ and a local chart $(W_{p_0}, \Psi)$ around $p_0,$ where $\Psi: W_ {p_0} \rightarrow B_{2n}(0,r_K)$ is biholomorphic with $\Psi(p_0)=0$  such that the two following conditions hold:  

$(i)$ we have 
\begin{align*} 
\|\Psi\|_1 \le c_1, \quad \|\Psi^{-1}\|_1 \le c_1,
\end{align*}

$(ii)$ there is a $\mathcal{C}^{3}$ map $h$ from $\overline{B}_n(0,1)$ to $\R^n$ so that $h(0)=Dh(0)=0,$ where $Dh$ denotes the differential of $h,$ and
 $$\Psi(K \cap W_{p_0}) \supset \{(\mathbf{x}, h(\mathbf{x})): \mathbf{x} \in \overline{B}_n(0,1)\},$$
 where the canonical coordinates on $\C^n= \R^n + i \R^n$ are denoted by $\mathbf{z}= \mathbf{x}+ i \mathbf{y},$ and 
\begin{align} \label{ine_C1alphanorm}
\|h\|_{3} \le c_1.
\end{align}
\end{lemma}
\begin{proof} We cover $X$ by a finite family of charts $(W_j, \Psi_j)$, where $W_j$ is an open subset of $X$ and $\Psi_j$ is a biholomorphic map from $W_j$ to the ball $B_{2n}(0,2).$ We choose these charts so that $\Psi_j^{-1}\big(B_{2n}(0,1)\big)$ also cover $X.$ This choice is independent of $p_0.$ Consider a chart $(W_{j_0},\Psi_{j_0})$ such that $p_0$ belongs to $\Psi_{j_0}^{-1}\big(B_{2n}(0,1)\big).$ Define $W_{p_0}:=W_{j_0}$ and $\Psi:= \Psi_{j_0}-\Psi_{j_0}(p_0).$   Let $\mathbf{z}=(z_1,\cdots,z_n)$ be the coordinates on $\C^n.$ Identify $K\cap W_{p_0}$ with $\Psi(K\cap W_{p_0})$ for convenience.  By the hypothesis on $K,$ we have 
$$T^{\R}_{\mathbf{z}}K+ i T^{\R}_{\mathbf{z}}K= \C^n.$$
This implies that there are a positive constant $c_1$ independent of $p_0$ and  a linear change of coordinates $\Psi'=(\Psi'_1,\cdots, \Psi'_n)$ of $\C^n$ such that 
\begin{align} \label{ine_linearchange}
\|\Psi'\|_1 \le c_1, \quad \|\Psi'^{-1}\|_1 \le c_1
\end{align}
and
$$\Psi'(p_0)=0 \quad  \text{ and } \quad \Psi' (T^{\R}_{p_0}K)=\{\Im \Psi'_k=0, 1 \le k \le n\},$$
where $T^{\R}_{\mathbf{z}}K$ is considered naturally  as an affine subspace of $\C^n.$ Replacing $\Psi$ by $\Psi' \circ \Psi$,  we can suppose that $\Psi(p_0)=0$ and 
$$\Psi_* (T^{\R}_{p_0}K)=\{\Im z_k=0, 1 \le k \le n\}=\{ \mathbf{z}= \mathbf{x}+ i 0\} \equiv \R^n.$$   
By rescaling $\Psi$ (by a constant independently of $p_0$)  if necessary, the submanifold 
$$K\cap W_{p_0} \cap \{\mathbf{z} \in \C^n: \Re \mathbf{z} \in \overline{B}_n(0,1)\}$$
is the graph of a $\mathcal{C}^{3}$ map $h=(h_1,\cdots, h_n)$ over $\overline{B}_n(0,1)$ of  $\R^n.$ By construction, we have $h(0)=Dh(0)=0.$ The compactness of $K$ and (\ref{ine_linearchange}) insure that there is a positive constant $c_1$ independent of $p_0$ such that $\|h\|_{3} \le c_1.$ The proof is finished. 
\end{proof}
From now on, we only use the local coordinates introduced in Lemma \ref{le_localcoordinates} and identify points in $W_{p_0}$ with those in $B_{2n}(0,r_K)$ via $\Psi.$   Property $(i)$ of Lemma \ref{le_localcoordinates} implies that the distance on $X$ is uniformly comparable with the Euclidean distance measured by the local coordinates given in Lemma \ref{le_localcoordinates}. Hence, in what follows, we make no distinction between these two distances. The estimate (\ref{ine_C1alphanorm}) implies that 
\begin{align}\label{ine_chuanC2}
 |h(\mathbf{x})| \le c_1 |\mathbf{x}|^{2}, \quad  |Dh(\mathbf{x})| \le c_1 |\mathbf{x}| \quad  \text{for } |\mathbf{x}| \le 1.
\end{align}

For each $\mathbf{z} \in B^*_{2n}(0,1)$ and $t \in (0,1],$ let $u_{\mathbf{z},t}$ be the map defined in  (\ref{defcuau_version8}).  Let $F$ and $c_0$ be the family of analytic discs and the constant respectively in Proposition \ref{pro_specialanalyticdisc}.  In order to construct an analytic disc half-attached to $K$, it suffices to find a  map 
$$U: \partial \D \rightarrow B_n(0,1) \subset \R^n,$$
which is H\"older continuous, satisfying the following Bishop-type equation 
\begin{align}\label{Bishoptype}
U_{\mathbf{z},t}(\xi)=t (\Re\mathbf{z}- \Im \mathbf{z}) - \mathcal{T}_1\big(h(U_{\mathbf{z},t}) \big)(\xi) - \mathcal{T}_1 u_{\mathbf{z},t}(\xi),
\end{align}      
where $\mathbf{z}$ and $t$ are parameters in $B^*_{2n}(0,1)$ and $(0,1)$ respectively. Indeed, suppose that (\ref{Bishoptype}) has a solution. For simplicity,  we use the same notation $U_{\mathbf{z},t}(z)$ to denote the harmonic extension of $U_{\mathbf{z},t}(\xi)$ to $\D.$ Let $P_{\mathbf{z},t}(z)$ be  the harmonic extension of $h\big(U_{\mathbf{z},t}(\xi)\big)$ to $\D.$ 
Define
$$F^h(z, \mathbf{z},t) := U_{\mathbf{z},t}(z)+ i  P_{\mathbf{z},t}(z)+ i  u_{\mathbf{z},t}(z)$$
which is a family of analytic discs parametrized by $(\mathbf{z},t).$ For any $\xi \in \partial^+ \D,$ the defining formula of $F^h$ and the fact that $u_{\mathbf{z},t} \equiv 0$ on $\partial^+ \D$ imply that  
$$F^h(\xi, \mathbf{z},t)=U_{\mathbf{z},t}(\xi)+ i  P_{\mathbf{z},t}(\xi)=U_{\mathbf{z},t}(\xi)+ i h\big(U_{\mathbf{z},t}(\xi)\big) \in K$$
by Property $(ii)$ of Lemma \ref{le_localcoordinates}. In other words, $F^h$ is half-attached to $K.$ Moreover we have  
$$F^h(1, \mathbf{z},t) =t (\Re\mathbf{z}- \Im \mathbf{z})+ i h(t\Re\mathbf{z}- t\Im \mathbf{z}).$$
In what follows, it is convenient to regard  $U_{\mathbf{z},t}(z)$ as a function of three variables $(z,\mathbf{z},t).$ 

\begin{proposition}\label{pro_Bishopequation}
There is a small positive number $t_1 \in (0,1)$ independent of $(\mathbf{z},p_0)$ so that  for any $t \in (0, t_1]$ and any $\mathbf{z} \in B^*_{2n}(0,1),$ the equation (\ref{Bishoptype}) has a unique solution $U_{\mathbf{z},t}$ such that $U_{\mathbf{z},t}(\xi)$ is  $\mathcal{C}^{2, \frac{1}{2}}$ in $\xi,$ the partial derivative  $D_{\mathbf{z}}U_{\mathbf{z},t} $ exists and  is $\mathcal{C}^{1,\frac{1}{2}}$ in $\xi \in \partial \D.$ Moreover, the two following  estimates hold:
\begin{align}\label{ine_danhgiachuancuaU}
\|U_{\mathbf{z},t}(\cdot)\|_{2, \frac{1}{2}} \le 4 c_0 t, \quad \|D_{\mathbf{z}} U_{\mathbf{z},t}(\cdot)\|_{1, \frac{1}{2}} \le 4 c_0 t |\mathbf{z}|^{-1}.
\end{align} 
\end{proposition}
Proposition \ref{pro_Bishopequation} except (\ref{ine_danhgiachuancuaU}) is a direct corollary of a more general result due to Tumanov, see \cite[Th. 4.19]{MerkerPorten2}.  Since we do not need the optimal regularity for $U_{\mathbf{z},t}$ (whereas it is the case for Tumanov's result), the proof is simpler. We will follow the presentation in \cite{MerkerPorten2}.  Firstly, we need the following preparatory lemma on the  norms of the H\"older spaces.

\begin{lemma} \label{le_danhgiag1g2chuanC11/2} Let  $g_1$ and $g_2$ be functions defined on $\partial \D$ with suitable differentiability. Then we have
\begin{align}\label{ine_danhgiatichC1/222}
\|g_1 g_2\|_{\frac{1}{2}} \le  \|g_1\|_{\frac{1}{2}} \|g_2\|_{\frac{1}{2}}
\end{align}
and
\begin{align}\label{ine_danhgiatichC1/2}
\|g_1 g_2\|_{1, \frac{1}{2}} \le  4 \|g_1\|_{1, \frac{1}{2}} \|g_2\|_{1, \frac{1}{2}}.
\end{align}
Moreover, there exists a positive constant $c(n)$ such that for any maps  $g_1,g_2$ from $\partial \D$ to $B_n(0,1)$ and any function $f$ on $B_n(0,1),$  we have
\begin{align}\label{ine_danhgiatichC1/233}
  \| f \circ g_1- f \circ g_2 \|_{1, \frac{1}{2}} \le  c(n)\{1+ \|g_2\|_{1, \frac{1}{2}}\} \|f\|_2\|g_1- g_2\|_{1, \frac{1}{2}}
\end{align}
and
\begin{align}\label{ine_danhgiatichC1/233'version2}
\| f \circ g_1\|_{1, \frac{1}{2}} \le c(n)\big\{ \|Df \circ g_1\|_0 \|g_1\|_{1, \frac{1}{2}} +\|f\|_2 \|g_1\|^2_{1, \frac{1}{2}}\big\}.
\end{align}
and
\begin{align}\label{ine_danhgiatichC1/233'}
\| f \circ g_1\|_{1, \frac{1}{2}} \le c(n)\{1+    \|g_1\|_{1, \frac{1}{2}}\} \|f\|_2\|g_1\|_{1, \frac{1}{2}}.
\end{align}
\end{lemma}

\begin{proof} Write 
$$g_1(\xi) g_2(\xi)- g_1(\xi') g_2(\xi')= g_1(\xi)\big(g_2(\xi)- g_2(\xi')\big)+ \big( g_1(\xi)- g_2(\xi') \big) g_2(\xi'),$$
for any $\xi,\xi' \in \partial \D.$ Using the last equality and the definition of the $\mathcal{C}^{\frac{1}{2}},$ one easily gets (\ref{ine_danhgiatichC1/222}).  
 Since  $D(g_1 g_2)= (Dg_1) g_2 + g_1 Dg_2,$ using (\ref{ine_danhgiatichC1/222}) and the definition of $\mathcal{C}^{1,\frac{1}{2}}$ gives
\begin{align*}
\|g_1 g_2\|_{1, \frac{1}{2}}& \le \|g_1 g_2\|_{1}+ \| D(g_1 g_2)\|_{\frac{1}{2}}\\ 
&\le \|g_1\|_1 \|g_2\|_{1}+ \|D g_1\|_{\frac{1}{2}}\|g_2\|_1+ \|g_1\|_1  \|Dg_2\|_{\frac{1}{2}} \le 4 \|g_1\|_{1, \frac{1}{2}} \|g_2\|_{1, \frac{1}{2}}.
\end{align*}
Hence, (\ref{ine_danhgiatichC1/2}) follows.

Now we prove (\ref{ine_danhgiatichC1/233}).  Let $g_1,g_2,f$ be as in the hypothesis of  (\ref{ine_danhgiatichC1/233}).  We have
$$D\big(f(g_1)- f(g_2) \big)= Df(g_1) D g_1 - D f(g_2) D g_2= Df(g_1)(D g_1 - D g_2)+ \big(Df(g_1)- Df(g_2)\big) D g_2.$$
Applying (\ref{ine_danhgiatichC1/222}) to the last sum shows that there exists a positive constant $c'(n)$ depending only on $n$ so that 
\begin{align*} 
\| f(g_1)- f(g_2) \|_{1, \frac{1}{2}} &\le \| f(g_1)- f(g_2) \|_1 + \| D\big(f(g_1)- f(g_2) \big) \|_{\frac{1}{2}} \\
\nonumber
 &\le  c'(n) \bigg\{\|f\|_1 \|g_1 -g_2\|_1+ \|f\|_2\|g_1 -g_2\|_0 \|g_2\|_1+\\
 \nonumber
& \quad  + \| D f(g_1)\|_{\frac{1}{2}} \| Dg_1 -Dg_2\|_{\frac{1}{2}} +  \|f\|_2 \|g_1 - g_2\|_{\frac{1}{2}}  \|g_2\|_{1, \frac{1}{2}} \bigg\}\\
\nonumber
& \le  c'(n)\{1+    3 \|g_2\|_{1, \frac{1}{2}}\} \|f\|_2\|g_1- g_2\|_{1, \frac{1}{2}}. 
\end{align*}
Hence,  (\ref{ine_danhgiatichC1/233}) follows by choosing $c(n)=3c'(n).$ The inequality (\ref{ine_danhgiatichC1/233'version2})  is  deduced by using the same method and  (\ref{ine_danhgiatichC1/233'}) is a direct consequence of  (\ref{ine_danhgiatichC1/233'version2}). The proof is finished.
\end{proof}

\begin{proof}[Proof of Proposition \ref{pro_Bishopequation}] Let  $C_{1,1/2}$ and $C_{2,1/2}$ be the constants $C_{k,\beta}$ appearing in (\ref{ine_chuancuaT}) with $k=1,2$ and $\beta=1/2$. Let $c(n)$ be the constant in Lemma \ref{le_danhgiag1g2chuanC11/2}.  Define 
$$t_1:= \big(40 c_0 c_1 c(n) \max\{C_{1,1/2},C_{2,1/2}\}+4 c_0 \big)^{-2}.$$
Fix  $t \in (0,t_1)$ and $\mathbf{z} \in B^*_{2n}(0,1).$  Let $\mathcal{A}$ be the set of $\mathcal{C}^{1,\frac{1}{2}}$ maps 
$$U: \partial \D \rightarrow \overline{B}_n(0, 4 c_0 t)$$ 
such that $\|U\|_{1, \frac{1}{2}} \le 4 c_0 t.$  Endow $\mathcal{A}$ with the $\mathcal{C}^{1,\frac{1}{2}}$-norm making it become a closed subset of a suitable Banach space. Define 
\begin{align} \label{eq_definitionofG}
G(U):=t (\Re\mathbf{z}- \Im \mathbf{z}) - \mathcal{T}_1\big(h(U) \big) - \mathcal{T}_1 u_{\mathbf{z},t}.
\end{align} 
We will show that $G$ is a well-defined self-map of $\mathcal{A}$ and  is a contraction. 

Let $U \in \mathcal{A}.$ Note that since $\mathbf{z} \in B^*_{2n}(0,1),$ we have $|\Re \mathbf{z} - \Im \mathbf{z}| \le 2.$ By (\ref{ine_chuancuaT}) and Proposition \ref{pro_specialanalyticdisc}, we get
\begin{align} \label{ine_chuanC11/2cuahU}
\|G(U)\|_{1, \frac{1}{2}} \le 2 t+  \| \mathcal{T}_1\|_{1, \frac{1}{2}} \|h(U)\|_{1, \frac{1}{2}}+ \|F(\cdot,\mathbf{z},t)\|_2 \le C_{1,1/2} \|h(U)\|_{1, \frac{1}{2}}+ 3 c_0 t
\end{align}
because we chose  $c_0 >1.$
The inequality (\ref{ine_danhgiatichC1/233'version2}) for $f= h$ and $g_1=U$ combined with (\ref{ine_chuanC2}) yields that
\begin{align} \label{ine_chuanC1/2cuahU}
\|h(U)\|_{1, \frac{1}{2}} \le 2 c(n) c_1  \|U\|^2_{1, \frac{1}{2}}.  
\end{align}
We deduce from (\ref{ine_chuanC1/2cuahU}) and (\ref{ine_chuanC11/2cuahU}) that 
\begin{align} \label{ine_chuanC11/2G(U)}
\|G(U)\|_{1, \frac{1}{2}} \le 2 c_1 c(n) C_{1,1/2}\|U\|^2_{1, \frac{1}{2}}+  3c_0 t \le  4 c_0 t
\end{align}
and  similarly using (\ref{eq_definitionofG}) gives
\begin{align} \label{ine_chuanC11/2G(U)contraction}
\|G(U)-G(U')\|_{1, \frac{1}{2}}& \le 2 c_1 c(n) C_{1,1/2}\|U-U'\|^2_{1, \frac{1}{2}} \\
\nonumber
&\le 16 c_0 c_1 c(n) C_{1,1/2} t\|U-U'\|_{1, \frac{1}{2}} \le   t^{1/2}\|U-U'\|_{1, \frac{1}{2}},
\end{align}
for any $U,U'\in \mathcal{A}.$ The inequality (\ref{ine_chuanC11/2G(U)}) shows that $G$ is well-defined. And the contractivity of $G$ follows from  (\ref{ine_chuanC11/2G(U)contraction}).  By the fixed point theorem, $G$ has a unique fixed point $U_{\mathbf{z},t} \in \mathcal{A}.$ In other words, the equation (\ref{Bishoptype}) has a unique solution $U_{\mathbf{z},t} \in \mathcal{C}^{1,\frac{1}{2}}(\D)$ and  $\|U_{\mathbf{z},t}\|_{1, \frac{1}{2}} \le 4 c_0 t.$ 

 Now we explain why $U_{\mathbf{z},t} \in \mathcal{C}^{2,\frac{1}{2}}(\partial \D).$  Define $\mathcal{A}'$ to be the subset of $\mathcal{A}$ consisting of $U$ with $\|U\|_{2,\frac{1}{2}} \le 4 c_0 t.$ Since $\mathcal{A}'$ is a closed subset with respect to the $\mathcal{C}^{2,\frac{1}{2}}$-norm in a suitable Banach space and $h,F \in \mathcal{C}^3,$ similar arguments as above applied to the $\mathcal{C}^{2,\frac{1}{2}}$-norm show that for $t$ small enough, precisely $t\in (0,t_1),$  $G$ is a self-contraction of $\mathcal{A}'$ and $U_{\mathbf{z},t}$ is the unique fixed point of $G.$ Note that in this argument, we need to use the constant $C_{2,1/2}$ which explains its presence in the definition of $t_1.$  Therefore,  $U_{\mathbf{z},t} \in \mathcal{C}^{2,\frac{1}{2}}$ and it satisfies 
$$\|U_{\mathbf{z},t}\|_{2,\frac{1}{2}} \le 4 c_0 t.$$   

Now we investigate the dependence of $U_{\mathbf{z},t}$ on the parameter $\mathbf{z}.$ Observe that if $U_{\mathbf{z},t}$ is differentiable in $\mathbf{z}$ and $D_{\mathbf{z}}U_{\mathbf{z},t}$ is at least $\mathcal{C}^{\beta}$ in $\xi$ for some $\beta \in (0,1),$ then by (\ref{Bishoptype}) and $\mathcal{C}^{\beta}$-boundedness of $\mathcal{T}_1$ we must have
\begin{align}\label{equ_phuongtrinhdaohamzU}
D_{\mathbf{z}}U_{\mathbf{z},t}= - \mathcal{T}_1\big(Dh(U_{\mathbf{z},t}) D_{\mathbf{z}}U_{\mathbf{z},t} \big) - \mathcal{T}_1 D_{\mathbf{z}}u_{\mathbf{z},t}.
\end{align}
This leads us to study the equation
\begin{align}\label{equ_phuongtrinhchodaohamz}
V= - \mathcal{T}_1\big( H \cdot V \big) - \mathcal{T}_1 D_{\mathbf{z}}u_{\mathbf{z},t},
\end{align}
where $H(\xi)= Dh(U_{\mathbf{z},t}(\xi))$ is a $\mathcal{C}^2$ matrix function in $\xi\in \partial \D.$ This equation is of the same type as (\ref{Bishoptype}). Since
\begin{align}\label{ine_danhgiadaohamtheozcuau1309}
\|D_{\mathbf{z}}u_{\mathbf{z},t}(\cdot)\|_2 \le  t c_0 |\mathbf{z}|^{-1} \quad (\text{see (\ref{ine_danhgiadaohamchoF})}),
\end{align}
and $H \in \mathcal{C}^2,$  the same arguments as above show that the equation (\ref{equ_phuongtrinhchodaohamz}) has a unique solution $V_{\mathbf{z},t}$ in $\mathcal{C}^{1, \frac{1}{2}}(\partial \D)$ with  
\begin{align}\label{ine_chuanC11/2cuaV}
\|V_{\mathbf{z},t}\|_{1, \frac{1}{2}} \le 4 c_0 |\mathbf{z}|^{-1} t.
\end{align} 
 Furthermore, if we define $V'^0_{\mathbf{z},t}= D_{\mathbf{z}}u_{\mathbf{z},t}$ and 
$$V'^{k+1}_{\mathbf{z},t}= - \mathcal{T}_1\big( H \cdot V'^k_{\mathbf{z},t} \big) - \mathcal{T}_1 D_{\mathbf{z}}u_{\mathbf{z},t}$$
 for $k \in \N^*,$ then 
\begin{align}\label{ine_hoitucuaVktoiV}
\|V'^k_{\mathbf{z},t}- V_{\mathbf{z},t}\|_{1, \frac{1}{2}} \le t^{k/2}|\mathbf{z}|^{-1}
\end{align}
thanks to the $t^{1/2}$-contractivity of the self-map defining  the recurrence relation of $V'^k_{\mathbf{z},t}.$ 

 We now relate $V_{\mathbf{z},t}$ to $U_{\mathbf{z},t}.$ Note that by (\ref{ine_danhgiadaohamchoF}), $u_{\mathbf{z},t} \in \mathcal{A}$ for $t \in (0,t_1).$   Let $\{U^k_{\mathbf{z},t}\}_{k \in \N}$ be the sequence in $\mathcal{A}$ defined by 
$$U^0_{\mathbf{z},t}= u_{\mathbf{z},t}, \quad U^k_{\mathbf{z},t}= G(U^{k-1}_{\mathbf{z},t}) \quad  \text{for} \quad k \ge 1.$$
Since $u_{\mathbf{z},t}$ is $\mathcal{C}^4$ in $(z,\mathbf{z})$ and $h \in \mathcal{C}^3$ and $\mathcal{T}_1$ is a linear  $\mathcal{C}^{2,\frac{1}{2}}$-bounded operator,   the functions $U^k_{\mathbf{z},t}$ are $\mathcal{C}^{2, \frac{1}{2}}$ in $(z,\mathbf{z})$  for all $k\ge 0.$ Define 
$$V^k_{\mathbf{z},t}:=D_{\mathbf{z}}U^k_{\mathbf{z},t} \in \mathcal{C}^{1, \frac{1}{2}} \quad  \text{in} \quad  (z, \mathbf{z}),$$ for $k\in \N.$
 By definition of $U^k_{\mathbf{z},t},$ the sequence $V^k_{\mathbf{z},t}$ is defined by the induction relation
$$V^{k+1}_{\mathbf{z},t}= - \mathcal{T}_1\big(Dh(U^{k}_{\mathbf{z},t})V^{k}_{\mathbf{z},t} \big) - \mathcal{T}_1 D_{\mathbf{z}}u_{\mathbf{z},t},$$
for $k\ge 0.$ Using (\ref{ine_danhgiadaohamtheozcuau1309}), the induction on $k$ and the above technique in the proof of (\ref{ine_chuanC11/2cuaV}),  we obtain that   
\begin{align} \label{ine_laidanhgiachuaC11/2cuaVk}
\|V^k_{\mathbf{z},t}\|_{1,\frac{1}{2}} \le  4 c_0 t|\mathbf{z}|^{-1}.
\end{align}
Since $G$ is $t^{1/2}$-contraction, we have 
\begin{align} \label{ine_danhgiahieusoUkvaU}
\|U^k_{\mathbf{z},t}- U_{\mathbf{z},t}\|_{1, \frac{1}{2}} \le  t^{k/2}.
\end{align}
We now compare $V'^k_{\mathbf{z},t}$ and $V^k_{\mathbf{z},t}.$ Their difference is 
\begin{align} \label{eq_hieugiuaVkvaVk'}
V^{k+1}_{\mathbf{z},t} -V'^{k+1}_{\mathbf{z},t}= - \mathcal{T}_1 \big[\big(Dh(U^{k}_{\mathbf{z},t})- Dh(U_{\mathbf{z},t})\big) V^k_{\mathbf{z},t} \big]  - \mathcal{T}_1 \big[ Dh(U_{\mathbf{z},t})(V^k_{\mathbf{z},t}- V'^k_{\mathbf{z},t})\big].
\end{align}
Applying (\ref{ine_danhgiatichC1/2}) and (\ref{ine_chuancuaT})  to each term in the right-hand side of (\ref{eq_hieugiuaVkvaVk'}) gives 
\begin{align} \label{ine_sosanhVkV'kmaichuaxong}
\|V^{k+1}_{\mathbf{z},t} -V'^{k+1}_{\mathbf{z},t}\|_{1,\frac{1}{2}} &\le 4C_{1, 1/2}\|Dh(U^k_{\mathbf{z},t})- Dh(U_{\mathbf{z},t})\|_{1, \frac{1}{2}} \|V^k_{\mathbf{z},t}\|_{1, \frac{1}{2}}+\\
\nonumber
&+ 4 C_{1,1/2}\|Dh(U_{\mathbf{z},t})\|_{1, \frac{1}{2}} \|V^k_{\mathbf{z},t}- V'^k_{\mathbf{z},t}\|_{1, \frac{1}{2}}.
\end{align}
The second term of the right-hand side of (\ref{ine_sosanhVkV'kmaichuaxong}) is less than or equal to 
$$8 c_1 c(n)C_{1,1/2}\|U_{\mathbf{z},t}\|_{1, \frac{1}{2}} \|V^k_{\mathbf{z},t}- V'^k_{\mathbf{z},t}\|_{1, \frac{1}{2}}$$
thanks to (\ref{ine_danhgiatichC1/233'}) and (\ref{ine_C1alphanorm}). By the first inequality of (\ref{ine_danhgiachuancuaU}) and (\ref{ine_laidanhgiachuaC11/2cuaVk}), the last quantity is less than or equal to 
$$32 c_0 c_1 c(n)C_{1,1/2} t |\mathbf{z}|^{-1}\|V^k_{\mathbf{z},t}- V'^k_{\mathbf{z},t}\|_{1, \frac{1}{2}}.$$
In a similar way, the first term of the right-hand side of (\ref{ine_sosanhVkV'kmaichuaxong}) is less than or equal to 
$$8 c_1C_{1,1/2}\|U^k_{\mathbf{z},t}- U_{\mathbf{z},t}\|_{1, \frac{1}{2}} \|V^k_{\mathbf{z},t}\|_{\frac{1}{2}}$$
thanks to (\ref{ine_danhgiatichC1/233}) and (\ref{ine_C1alphanorm}). By (\ref{ine_laidanhgiachuaC11/2cuaVk}) and (\ref{ine_danhgiahieusoUkvaU}), the last quantity is also less than or equal to 
$$32 c_0 c_1 c(n)C_{1,1/2} t^{(k+2)/2}|\mathbf{z}|^{-1}.$$
Hence, we just proved that 
$$\|V^{k+1}_{\mathbf{z},t} -V'^{k+1}_{\mathbf{z},t}\|_{1,\frac{1}{2}} \le  32 c_0 c_1 c(n)C_{1,1/2}\bigg[ t^{(k+2)/2}|\mathbf{z}|^{-1}+ t    \|V^k_{\mathbf{z},t}- V'^k_{\mathbf{z},t}\|_{1, \frac{1}{2}} \bigg]. $$
By induction on $k$ and   the last inequality,  one easily deduces that  
$$\|V^{k+1}_{\mathbf{z},t} -V'^{k+1}_{\mathbf{z},t}\|_{1, \frac{1}{2}} \le t^{k/2} |\mathbf{z}|^{-1}$$
 for all $k \in \N.$ Combining with the fact that $V'^k_{\mathbf{z},t} \rightarrow V_{\mathbf{z},t},$ we get $V^k_{\mathbf{z},t} \rightarrow V_{\mathbf{z},t}.$ Integrating the last limit with respect to $\mathbf{z}$, one sees that $U_{\mathbf{z},t}$ is differentiable on $\mathbf{z}$ and $D_{\mathbf{z}}U_{\mathbf{z},t}= V_{\mathbf{z},t}.$ In particular, $D_{\mathbf{z}}U_{\mathbf{z},t}$ belongs to $\mathcal{C}^{1,\frac{1}{2}}(\partial \D).$ The proof is finished.
\end{proof}

 Let $t$ be  a  real number  in $ (0, t_1)$ as in Proposition \ref{pro_Bishopequation}. Define the map 
$$\Phi^h: B_{2n}(0,1) \rightarrow \C^n$$
 by putting $\Phi^h(0)=0$ and $\Phi^h(\mathbf{z})= F^h(1-|\mathbf{z}|+ i|\mathbf{z}|, \mathbf{z},t)$ for $\mathbf{z}\not = 0.$  Let $\Phi$ and $g$ be the maps defined in the proof of Proposition \ref{pro_specialanalyticdisc}.  Recall that $\Phi(\mathbf{z})=F(1-|\mathbf{z}|+ i|\mathbf{z}|, \mathbf{z},t)$ and $\Phi(\mathbf{z})= t \mathbf{z}+ g(\mathbf{z})$ and $$\|g\|_{1, B^*_{2n}(0,2r_0)} \le t/4, \quad  g(0)=0.$$
We want to prove that  $\Phi^h\big( B_{2n}(0,1)\big)$ contains an open neighborhood of $0$ just as what we did for $\Phi.$ To this end, we compare below these two maps and their derivatives.

\begin{lemma} \label{le_differentiabilityofPhi} There is a positive constant $c_2$ independent of $\mathbf{z}, p_0$ and of $t$ such that  for all $\mathbf{z} \in B^*_{2n}(0,1),$   we have 
\begin{align}\label{ine_estimateonphi}
\big|\Phi^h(\mathbf{z})- \Phi(\mathbf{z}) \big| \le c_2 t^{2}|\mathbf{z}|,
\end{align}
and 
\begin{align} \label{ine_estimateonphi2}
\big|D_{\mathbf{z}}\Phi^h(\mathbf{z})- D_{\mathbf{z}}\Phi(\mathbf{z})\big| \le c_2 t^2.
\end{align}
\end{lemma} 

\begin{proof} Let $\mathbf{z} \in B^*_{2n}(0,1).$ Let $z^*:= 1-|\mathbf{z}|+ i |\mathbf{z}|.$   We deduce from (\ref{Bishoptype}) and the definition of $F^h$ that 
$$E(z,\mathbf{z},t):=F^h(z, \mathbf{z},t)- F(z, \mathbf{z},t)= - \mathcal{T}_1(P_{\mathbf{z},t})(z) + i P_{\mathbf{z},t}(z)$$
is a holomorphic map in $z$. Substituting $z$ by $z^*$ in the last equality gives 
\begin{align}\label{eq_nhanhlenchanquaroi}
\Phi^h(\mathbf{z})-\Phi(\mathbf{z}) = E(z^*, \mathbf{z},t).
\end{align}
Recall that  $P_{\mathbf{z},t}(\xi)=h\big(U_{\mathbf{z},t}(\xi)\big),$ for $\xi \in \D.$  By (\ref{ine_chuanC1/2cuahU}), (\ref{ine_danhgiachuancuaU}) and (\ref{ine_danhgiachuaCkcuauvoibien}), we have
\begin{align*} 
\|P_{\mathbf{z},t}(\cdot)\|_{1, \frac{1}{2},\overline{\D}} \lesssim  \|P_{\mathbf{z},t}(\cdot)\|_{1, \frac{1}{2},\partial\D} \lesssim   t^{2}.
\end{align*}
This yields that
\begin{align} \label{ine_danhgiadaohamcuaEtheoz}
\|E(\cdot, \mathbf{z},t)\|_{1,\frac{1}{2},\overline{\D}} \lesssim  \|\mathcal{T}_1(P_{\mathbf{z},t})\|_{1,\frac{1}{2},\overline{\D}}+  \|P_{\mathbf{z},t}\|_{1,\frac{1}{2},\overline{\D}} \lesssim \| P_{\mathbf{z},t} \|_{1,\frac{1}{2},\overline{\D}} \lesssim  t^2.
\end{align}
We have
$$E(1, \mathbf{z},t)=i h\big(U_{\mathbf{z},t}(1) \big)= i h( t \Re \mathbf{z}-t \Im \mathbf{z})$$
which is of modulus less than or equal to $t^2 |\Re \mathbf{z}-\Im \mathbf{z}|^2 \le 2t^2|\mathbf{z}|^2$ by (\ref{ine_chuanC2}). Using the last inequality and (\ref{ine_danhgiadaohamcuaEtheoz}), one has 
$$|E(z^*, \mathbf{z},t)| \le |E(z^*, \mathbf{z},t)- E(1, \mathbf{z},t)|+ | E(1, \mathbf{z},t)| \le  \| E(\cdot, \mathbf{z},t) \|_{1,\overline{\D}} |1-z^*| + 2 t^2|\mathbf{z}|^2  \lesssim t^2|\mathbf{z}|.$$
Using this and (\ref{eq_nhanhlenchanquaroi}), one gets (\ref{ine_estimateonphi}).

Differentiating  (\ref{eq_nhanhlenchanquaroi}) gives
\begin{align} \label{eq_tinhdaohamtheozcuaphi}
D_{\mathbf{z}}\Phi^h(\mathbf{z})- D_{\mathbf{z}}\Phi(\mathbf{z})=  D_{\mathbf{z}} E(z^*, \mathbf{z},t) + D_z E(z^* , \mathbf{z},t)\begin{bmatrix} -D_{\mathbf{z}}|\mathbf{z}|\\ \quad D_{\mathbf{z}}|\mathbf{z}| \end{bmatrix}.
\end{align}
By (\ref{ine_danhgiadaohamcuaEtheoz}), we have 
$$\big|D_z E(z^*, \mathbf{z},t)D_{\mathbf{z}}|\mathbf{z}| \big| \lesssim t^2 \big|D_{\mathbf{z}} |\mathbf{z}| \big|  \lesssim t^2.$$
 Hence it remains  to estimate the first term in the right-hand side of (\ref{eq_tinhdaohamtheozcuaphi}).  Observe that 
$$ D_{\mathbf{z}} E(z, \mathbf{z},t)= - \mathcal{T}_1\big(D_{\mathbf{z}}P_{\mathbf{z},t}\big)(z)+ i  D_{\mathbf{z}}P_{\mathbf{z},t}(z),$$
for all $z \in \overline{\D}.$ 
Let $H(\xi)=D h(U_{\mathbf{z},t}(\xi))$ the function defined in the proof of   Proposition \ref{pro_Bishopequation}. By definition of $P_{\mathbf{z},t},$ we have 
$$D_{\mathbf{z}} P_{\mathbf{z},t}(\xi)=H(\xi) D_{\mathbf{z}} U_{\mathbf{z},t}(\xi).$$
Using (\ref{ine_danhgiatichC1/2}) together with (\ref{ine_danhgiachuancuaU}) gives
\begin{align*}   
\| D_{\mathbf{z}}P_{\mathbf{z},t}\|_{1, \frac{1}{2}, \overline{\D}} \le \| D_{\mathbf{z}}P_{\mathbf{z},t}\|_{1, \frac{1}{2}, \partial\D} \le  4\|H\|_{1, \frac{1}{2},\partial \D}\|D_{\mathbf{z}}U_{\mathbf{z},t}\|_{1, \frac{1}{2}, \partial \D } \lesssim t^2 |\mathbf{z}|^{-1}. 
\end{align*}
We have $D_{\mathbf{z}} E(1, \mathbf{z},t)= i D_{\mathbf{z}}   h( t \Re \mathbf{z}-t \Im \mathbf{z})$ which is clearly of absolute value $ \lesssim t^2$ by (\ref{ine_chuanC2}). Combining this with (\ref{ine_chuancuaT}) yields that 
\begin{align*} 
|D_{\mathbf{z}} E(z^*, \mathbf{z},t)| &= |D_{\mathbf{z}} E(z^*, \mathbf{z},t)- D_{\mathbf{z}} E(1, \mathbf{z},t)|+ |D_{\mathbf{z}} E(1, \mathbf{z},t)|  \\
&\le |1- z^*| \| D_{\mathbf{z}} E(\cdot, \mathbf{z},t)\|_{1, \overline{\D}}+ t^2 \lesssim  |1- z^*| \| D_{\mathbf{z}} P_{\mathbf{z},t}\|_{1,\overline{\D}}+ t^2 \lesssim t^{2}. 
\end{align*}
The proof is finished.
\end{proof}

\begin{lemma} \label{le_implicittheorem} Let $r_0$ be the constant in Proposition \ref{pro_specialanalyticdisc}. There is a positive number $t_2<t_1$ independent of $p_0$ and of $\mathbf{z}$ such that for any $t \in (0,t_2],$ the set $\Phi^h \big(B_{2n}(0, r_0)\big)$ contains the ball $B_{2n}(0,  r_0 t/2).$
\end{lemma}
\begin{proof}  Define $g^h(\mathbf{z}):= \Phi^h(\mathbf{z})- t \mathbf{z}.$ We have  $g^h= \Phi^h - \Phi +g.$ By Lemma \ref{le_differentiabilityofPhi}, $\Phi^h -\Phi$ is $t/4$-Lipschitz for $t\le t_2: = \min\{t_1, (4c_2)^{-1}\}.$ Combining with the fact that $g$ is $t/4$-Lipschitz on $\overline{B}_{2n}(0,2r_0)$ implies that $g^h$ is $t/2$-Lipschitz on $\overline{B}_{2n}(0,2r_0)$ for $t \in (0,t_2).$  Now, an application of Lemma \ref{le_implicitfunction2} to $\Phi^h= t\Id + g^h$ gives the desired result. The proof is finished.  
\end{proof}

\begin{proof}[Proof of Proposition \ref{pro_familydiscK} for the case without singularity.] In our chosen local coordinates \\ around $p_0,$ 
we have $p_0=0$ and $p=\mathbf{z}.$  Let $t \in (0,t_2)$ be as in Lemma \ref{le_implicittheorem}. For any $\mathbf{z} \in B^*_{2n}(0,  r_0 t/2),$ there is $\mathbf{z}^* \in B_{2n}(0,r_0)$ for which $\Phi^h(\mathbf{z}^*)=\mathbf{z}.$ We deduce from (\ref{ine_estimateonphi}) that 
$$|\mathbf{z}- \Phi(\mathbf{z}^*)| \le c_2 t^2 |\mathbf{z}^*| \le \frac{ t| \mathbf{z}^*|}{4} \cdot$$
At the end of the proof of  Proposition \ref{pro_specialanalyticdisc}, we proved that $|\Phi(\mathbf{z}^*)| \ge t |\mathbf{z}^*|/2.$ This implies that $|\mathbf{z}^*| \le 4 |\mathbf{z}|/ t.$ Let $f(z):= F^h(z,\mathbf{z}^*,  t_2)$ and $z^*=1- |\mathbf{z}|+ i |\mathbf{z}|.$ The last inequality implies that 
$$|1- z^*| \le 2|\mathbf{z}^*| \le 8|\mathbf{z}|/t.$$ 
The analytic disc $f$ clearly satisfies all requirements in Proposition \ref{pro_familydiscK}. 

Now, we explain how to obtain the desired analytic discs when $\dim_{\R} K >n$.   Since we only consider small discs near $K,$ it is enough to work in a small chart and identify $K$ with  a submanifold of $B_{2n}(0,1)$ with $0 \in K.$  Choose a real linear space $A$ through $0$ such that $A$ intersects $K \cap B_{2n}(0,2r)$ transversally at a generic manifold of dimension $n,$ where $r>0$ is a positive number.  We can choose $r$ small enough such that this property also holds for any linear subspace $A'$ parallel to $A$ which intersects  $K \cap B_{2n}(0,2r).$  Let $p_0 \in K \cap B_{2n}(0,r)$ and $p \in B_{2n}(0,1)$ close to $p_0.$ Let $K'$ be the intersection of $K \cap B_{2n}(0,2r)$ with the linear space $A'$ through $p_0$ and parallel to $A.$ The construction in the last subsections can be applied to $(K',X, p_0,p)$ without changes. We obtain analytic discs half-attached to $K',$ hence half-attached to $K,$ with the properties described in Proposition \ref{pro_familydiscK}. The proof  is finished. 
\end{proof}

\subsection{The case where $K$ has singularity} \label{subsec_withboundary}
We treat the case where $K$ is a compact generic nondegenerate $\mathcal{C}^5$-piecewise submanifold of $X$. Actually, $\mathcal{C}^4$-differentiability is enough for our proof but in order to avoid some involvedly technical points, we will use $\mathcal{C}^5$-differentiability. 

The case of higher dimension will be treated at the end of this subsection also by considering generic submanifolds of $K.$ The following is an analogue of Lemma \ref{le_localcoordinates}. 
 
\begin{lemma} \label{le_localcoordinates2} There exist  constants $c_1, r_K >1$ depending only on $(K,X)$ and a local chart $(W_{p_0}, \Psi)$ around $p_0,$ where $\Psi: W_ {p_0} \rightarrow B_{2n}(0,r_K)$ is biholomorphic with $\Psi(p_0)=0$  such that the two following conditions hold:  

$(i)$ we have 
\begin{align*} 
\|\Psi\|_1 \le c_1, \quad \|\Psi^{-1}\|_1 \le c_1,
\end{align*}

$(ii)$ there is a $\mathcal{C}^5$ map $h$ defined on $\overline{B}_n(0,1)$ with $h(0)=Dh(0)=0,$ so that
 $$\Psi(K \cap W_{p_0}) \supset \big\{(\mathbf{x}, h(\mathbf{x})): \mathbf{x} \in (\R^+)^n \cap \overline{B}_n(0,1) \big\}$$ 
and 
\begin{align} \label{ine_C1alphanorm2}
\|h\|_{5} \le c_1.
\end{align}
\end{lemma}
  
\begin{proof} Firstly, observe that by definition of $K,$ through every point $p$ on the singularity of $K,$ there exist a local chart $W$ of $p$ in $X$ so that $K \cap W$ is the intersection of $W$ with a finite union of convex polyhedra of dimension $n$ in $\R^{2n}.$ Here we identified $W$ and $K \cap W$ with their images in $\R^{2n}.$ Let $K''$ be one of these convex polyhedra containing $p$. Let $K'$ be the intersection of $W$ with the linear subspace of $\R^{2n}$ supporting $K''.$ Since $K'' \cap W$ is a generic submanifold of $W$ (because $K\cap W$ is so),  $K'$ is a  generic smooth submanifold without boundary of $W$  by shrinking $W$ if necessary. Note that $p \in K'' \subset K \cap W \cap K'$ and $\dim K''= \dim K'=n.$  

The above observation shows that we can cover $K$ by a finite number of  holomorphic charts $(W_j, \Psi_j)$ of $X$ such that there are  generic  $n$-dimensional submanifolds $K'_j$ without boundary of $W_j$ and subsets $K''_j$ of $K'_j \cap K \cap W_j$ diffeomorphic to the intersection of $W_j$ with a convex polyheron of dimension $n$ via suitable local charts of $X.$  Without loss of generality, we can suppose that $\Psi_j$ are biholomorphisms from $W_j$ to $B_{2n}(0,2)$ and  the open sets $\Psi_j^{-1}\big(B_{2n}(0,1)\big)$ also cover $K.$ 

Consider a chart $(W_{j_0}, \Psi_{j_0})$ such that $p_0 \in \Psi_{j_0}^{-1}\big(B_{2n}(0,1)\big).$ As above, we can suppose that $p_0 \in K''_{j_0}.$  Put $W_{p_0}:= W_{j_0}.$ Using the fact that $K'_{j_0}$ is a generic $n$-dimensional smooth submanifold of $W_{j_0}$ and  arguing as in Lemma \ref{le_localcoordinates}, we see that by replacing $\Psi_{j_0}$ by the composition of $\Psi_{j_0}$  with a suitable affine linear map of $\C^n,$ one obtain $\Psi_{j_0}(p_0)=0$ and $K'_{j_0}$ contains the graph of a $\mathcal{C}^5$ map $h(\mathbf{x})$ over $\overline{B}_n(0,1)$ and $h(0)=Dh(0)=0.$   

By the choice of $K''_{j_0}$ and rescaling $\Psi_{j_0}$ if necessary, there exist  $\mathcal{C}^5$ functions $\tau_j(\mathbf{x})$ defined on an open neighborhood of $\overline{B}_n(0,1)$ with $1\le j \le n$ such that
$$K \cap W_{p_0} \supset K''_{j_0}  \supset \{(\mathbf{x}, h(\mathbf{x})): \tau_j(\mathbf{x}) \ge 0\,  \text{ for all } 1 \le j \le n \}$$
and the Jacobian $D(\tau_1,\cdots, \tau_n)/ D\mathbf{x}$ is of maximal rank in $B_n(0,1)$. Write $\mathbf{x}=(x_1,\cdots,x_n).$ Since every linear change of coordinates in $\R^n$ can be extended naturally to be a complex linear change of $\C^n,$ using a suitable complex linear change of coordinates in $\C^n$ allows one to assume that tangent space of $\{\tau_j=0\}$ at $0$ is $\{x_j=0\}$ for $1 \le j \le n.$ Notice that the distortion caused by the change of coordinates is bounded independently of $p_0.$   For $\mathbf{x}\in \overline{B}_n(0,1),$ write 
\begin{align}\label{ine_taujTaylor}
\tau_j(\mathbf{x})=\tau_j(0)+ \sum_{l=1}^n\partial_{x_l}\tau_j(0)x_l+ O(|\mathbf{x}|^2) \ge x_j - \|\tau_j\|_2\sum_{l=1}^n |x_l|^2.
\end{align}
Put $C= \sup_{1\le j \le n} \|\tau_j\|_2.$ Define
$$Q_n=\big\{\mathbf{x} \in \R^n: x_j \ge \frac{1}{3n}\sum_{l=1}^n x_l \quad \forall \, 1\le j \le n  \big\} \subset (\R^+)^n.$$
For $\mathbf{x} \in Q_n$ with $|\mathbf{x}| \le \frac{1}{3nC},$  the inequality (\ref{ine_taujTaylor}) yields that 
$$\tau_j(\mathbf{x}) \ge x_j - C \frac{1}{3nC} \sum_{l=1}^n x_l \ge x_j - \frac{1}{3n} \sum_{l=1}^n x_l \ge 0,$$
for all $1\le j \le n.$ We deduce that 
$$K \cap W_{p_0}  \supset \big\{(\mathbf{x}, h(\mathbf{x})): \mathbf{x} \in Q_n \cap \overline{B}_{n}(0,\frac{1}{3nC}) \big\}.$$
The composition of a suitable linear change of coordinates in $\R^n$ with a dilation in $\R^n$  will map $Q_n$ onto $(\R^+)^n$ and map $\overline{B}_{n}(0, \frac{1}{3nC})$  onto a neighborhood of $\overline{B}_{n}(0,1).$ This map can be extended to be a holomorphic change of coordinates $\Psi'$ in $\C^n.$  Composing $\Psi_{j_0}$ with $\Psi'$, we get the desired change of coordinates and the property $(ii).$ The proof is finished.
\end{proof}
Let $K_h:=\big\{(\mathbf{x}, h(\mathbf{x})): \mathbf{x} \in  \overline{B}_n(0,1) \big\}$ which is a  $\mathcal{C}^5$ submanifold of $B_{2n}(0,1).$  Property $(ii)$ of Lemma \ref{le_localcoordinates2} implies that 
\begin{align}\label{ine_chuanC3cuahdanhgialuythua3}
|h(\mathbf{x})| \le c_1 |\mathbf{x}|^2, \quad |Dh(\mathbf{x})| \le c_1 |\mathbf{x}|,   \quad \text{for } |\mathbf{x}| \le 1.
\end{align}

To establish the desired family of analytic discs in this  context, we follow the same strategy as in the previous case.  Let $F'_{\boldsymbol{\tau}}, u'_{\mathbf{z},t,\boldsymbol{\tau}}, c_0$  be the maps and the constant defined in (\ref{de_dinhnghiacuaF'}), (\ref{de_dinhnghiacuau'}) and (\ref{ine_danhgiadaohamchoF'tau}).  As in the last subsection, consider the  following Bishop-type equation 
\begin{align}\label{Bishoptype2}
U'_{\mathbf{z},t,\boldsymbol{\tau}}(\xi)= 2t(|\mathbf{z}|,\cdots, |\mathbf{z}|) - \mathcal{T}_1\big(h(U'_{\mathbf{z},t,\boldsymbol{\tau}}) \big)(\xi) - \mathcal{T}_1 u'_{\mathbf{z},t,\boldsymbol{\tau}}(\xi),
\end{align}      
for $\mathbf{z}  \in B^*_{2n}(0,\frac{1}{2n})$, $ t \in (0,1]$ and $\boldsymbol{\tau} \in B_n(0,2).$ For simplicity,  we use the same notation $U'_{\mathbf{z},t,\boldsymbol{\tau}}(z)$ to denote the harmonic extension of $U'_{\mathbf{z},t,\boldsymbol{\tau}}(\xi)$ to $\D.$ Let $P'_{\mathbf{z},t,\boldsymbol{\tau}}(z)$ be  the harmonic extension of $h\big(U'_{\mathbf{z},t,\boldsymbol{\tau}}(\xi)\big)$ to $\D.$ 
If $U'_{\mathbf{z},t,\boldsymbol{\tau}}$ is a solution of (\ref{Bishoptype2}) which is at least H\"older continuous, then
$$F'^h_{\boldsymbol{\tau}}(z, \mathbf{z},t) := U'_{\mathbf{z},t,\boldsymbol{\tau}}(z)+ i  P'_{\mathbf{z},t,\boldsymbol{\tau}}(z)+ i  u'_{\mathbf{z},t,\boldsymbol{\tau}}(z)$$
 is clearly a family of analytic discs half-attached to $K_h$ and  
$$F'^h_{\boldsymbol{\tau}}(1, \mathbf{z},t) =2t(|\mathbf{z}|,\cdots, |\mathbf{z}|)+ i h(2t|\mathbf{z}|,\cdots, 2t|\mathbf{z}|) \in (\R^+)^n.$$
Our goal is to obtain a stronger property that  $F'^h_{\boldsymbol{\tau}}$ is  $I$-attached to $K \subset K_h,$ for some interval $I \subset \partial \D$ containing $1.$ In view of $(ii)$ of Lemma \ref{le_localcoordinates2},  it suffices to prove that 
$$U'_{\mathbf{z},t,\boldsymbol{\tau}}(\xi) \ge 0 \quad  \text{ for $\xi \in I$}.$$
Here for any $r \in \R$ and $\mathbf{v} \in \R^n,$ we write $\mathbf{v} \ge r$ to indicate that each component of $\mathbf{v}$ is greater than or equal to $r.$ A similar convention is applied to $\mathbf{v} \le r.$

\begin{proposition}\label{pro_Bishopequation2} There is a positive number $t_1\in (0,1)$ independent of $(p_0,\mathbf{z}, \boldsymbol{\tau})$ so that  for any $t \in (0, t_1)$  and any $\mathbf{z} \in B^*_{2n}(0,\frac{1}{2n}),$ the equation (\ref{Bishoptype2}) has a unique solution $U'_{\mathbf{z},t,\boldsymbol{\tau}}$ such that $U'_{\mathbf{z},t,\boldsymbol{\tau}}(\xi)$ is  $\mathcal{C}^{4, \frac{1}{2}}$ in $\xi,$ for $1\le j\le 4$ the differential $D^j_{(\mathbf{z},\boldsymbol{\tau})}U'_{\mathbf{z},t,\boldsymbol{\tau}} $ exists and is $\mathcal{C}^{4-j,\frac{1}{2}}$ in $\xi \in \partial \D.$ Moreover, the following  estimates hold:
\begin{align}\label{ine_danhgiachuancuaU'2}
\|U'_{\mathbf{z},t,\boldsymbol{\tau}}(\cdot)\|_{4, \frac{1}{2}} \le 4 c_0 t,\quad \|D^{j'}_{\boldsym{\tau}}U'_{\mathbf{z},t,\boldsymbol{\tau}}(\cdot)\|_{4-j, \frac{1}{2}} \le 4 c_0 t, \quad \| D^j_{\boldsym{\tau}} D_{\mathbf{z}} U'_{\mathbf{z},t,\boldsymbol{\tau}}(\cdot)\|_{3-j, \frac{1}{2}} \le 4 c_0 |\mathbf{z}|^{-1}t,
\end{align} 
for $j=0,1,2,3$ and $j'=1,2,3,4.$
\end{proposition}
\begin{proof}  
By (\ref{de_dinhnghiacuau'})-(\ref{eq_daohamtheotautai0cuau'tau}), we see that  the arguments in the proof of Proposition \ref{pro_Bishopequation} still work for this case. Hence, the proof is finished.
\end{proof}
From now on, let $t_1$ be the constant in Proposition \ref{pro_Bishopequation2} and let $t \in (0,t_1).$ Let $U'_{\mathbf{z},t,\boldsymbol{\tau}}$ be the solution of  (\ref{Bishoptype2}) described in Proposition \ref{pro_Bishopequation2}. For $\xi\in \partial \D,$ write $\xi= e^{i\theta}$ with $\theta\in [-\pi,\pi).$

\begin{lemma}\label{le_danhgiadaohamcuaP'tau} There exists a constant $c_2$ independent of $(p_0,\mathbf{z},t,\boldsymbol{\tau})$ so that for any $(\mathbf{z},t,\boldsymbol{\tau}),$ we have
\begin{align}\label{ine_danhgiachuancuaP'2}
\|P'_{\mathbf{z},t,\boldsymbol{\tau}}(\cdot)\|_{4, \frac{1}{2}, \overline{\D}} \le c_2 t^2, \quad \|D^j_{\boldsymbol{\tau}} P'_{\mathbf{z},t,\boldsymbol{\tau}}(\cdot)\|_{4-j, \frac{1}{2},\overline{\D}} \le c_2 t^2,
\end{align} 
for $j=1,2,3,4$ and
\begin{align}\label{ine_danhgiachuancuaP'2theoz}
\|D^j_{\boldsymbol{\tau}} D_{\mathbf{z}}P'_{\mathbf{z},t,\boldsymbol{\tau}}(\cdot)\|_{3-j, \frac{1}{2},\overline{\D}} \le c_2 t^2|\mathbf{z}|^{-1},
\end{align}
for $j=0,1,2,3.$
\end{lemma}

\begin{proof} In view of  (\ref{ine_danhgiachuaCkcuauvoibien}), it is enough to estimate the norms of  $P'_{\mathbf{z},t,\boldsymbol{\tau}}$ and  $D^j_{(\mathbf{z},\boldsymbol{\tau})} P'_{\mathbf{z},t,\boldsymbol{\tau}}$ on $\partial \D,$ for $j=1,2.$ Since  $P'_{\mathbf{z},t,\boldsymbol{\tau}}(\xi)= h\big(U'_{\mathbf{z},t,\boldsymbol{\tau}}(\xi)\big)$ on $\partial \D,$ we have 
$$\partial_{\xi}P'_{\mathbf{z},t,\boldsymbol{\tau}}(\xi)= Dh\big(U'_{\mathbf{z},t,\boldsymbol{\tau}}(\xi)\big) \partial_{\xi}U'_{\mathbf{z},t,\boldsymbol{\tau}}(\xi).$$  
This combined with (\ref{ine_chuanC3cuahdanhgialuythua3}) and (\ref{ine_danhgiachuancuaU'2}) yields that 
$$|\partial_{\xi}P'_{\mathbf{z},t,\boldsymbol{\tau}}(\xi)| \le c_1 |U'_{\mathbf{z},t,\boldsymbol{\tau}}(\xi)| \, |\partial_{\xi}U'_{\mathbf{z},t,\boldsymbol{\tau}}(\xi)| \le 4 c_0 c_1 t^2.$$
By similar arguments, we also have $|\partial^j_{\xi}P'_{\mathbf{z},t,\boldsymbol{\tau}}(\xi)| \lesssim  t^2$ with $j=1,2.$ Hence, we obtain the first inequality in (\ref{ine_danhgiachuancuaP'2}).  For the proofs of  the remaining inequalities,  observe that $D^j_{(\mathbf{z},\boldsymbol{\tau})} P'_{\mathbf{z},t,\boldsymbol{\tau}}$ is the harmonic extension of $D^j_{(\mathbf{z},\boldsymbol{\tau})} h \big( U'_{\mathbf{z},t,\boldsymbol{\tau}}(\cdot) \big)$ to $\D.$ Hence, analogous reasoning gives the desired result. The proof is finished. 
\end{proof}

 \begin{lemma} \label{le_daohamtheoxU'tauzt}   We always have 
$$|\partial_{x}U'_{\mathbf{z},t,\boldsymbol{\tau}}(1)| \lesssim t^2 |\mathbf{z}|,\quad | D_{\boldsym{\tau}}\partial_{x}U'_{\mathbf{z},t,\boldsymbol{\tau}}(1)| \lesssim t^2 |\mathbf{z}| \quad \text{and} \quad | D_{\mathbf{z}}\partial_{x}U'_{\mathbf{z},t,\boldsymbol{\tau}}(1)| \lesssim t^2.$$ 
 \end{lemma}
 \begin{proof} By (\ref{eq_daohamvhailantheotheta''}), one has
  $$\partial_{y}P'_{\mathbf{z},t,\boldsymbol{\tau}}(1)=\partial_{\theta} h\big(U'_{\mathbf{z},t,\boldsymbol{\tau}}(e^{i\theta})\big)|_{\theta=0} =Dh\big(U'_{\mathbf{z},t,\boldsymbol{\tau}}(1)\big) \partial_{\theta} U'_{\mathbf{z},t,\boldsymbol{\tau}}(1).$$
Since $U'_{\mathbf{z},t}(1)=2t(|\mathbf{z}|, \cdots, |\mathbf{z}|),$ using (\ref{ine_chuanC3cuahdanhgialuythua3}) and  (\ref{ine_danhgiachuancuaU'2}), we get
$$|\partial_{y}P'_{\mathbf{z},t,\boldsymbol{\tau}}(1) | \lesssim t^2 |\mathbf{z}|.$$
The Cauchy-Riemann equations for $F'^h_{\boldsymbol{\tau}}$ give 
\begin{align*}
\partial_{x}U'_{\mathbf{z},t,\boldsymbol{\tau}}(z)=\partial_{y}P'_{\mathbf{z},t,\boldsymbol{\tau}}(z)+ \partial_{y}u'_{\mathbf{z},t,\boldsymbol{\tau}}(z), 
\end{align*}
for all $z\in \overline{\D}.$   Substituting $z$ by $1$ in the last equation, 
 we obtain
\begin{align*}
\partial_{x}U'_{\mathbf{z},t,\boldsymbol{\tau}}(1)=\partial_{y}P'_{\mathbf{z},t,\boldsymbol{\tau}}(1) + \partial_{y}u'_{\mathbf{z},t,\boldsymbol{\tau}}(1)=\partial_{y}P'_{\mathbf{z},t,\boldsymbol{\tau}}(1)= O(t^2 |\mathbf{z}|).
\end{align*}
Hence, the first desired inequality follows.  As to the second one, by differentiating the last inequality with respect to $\boldsym{\tau},$ we get
\begin{align} \label{eq_daohamboldtau26/12lauqua}
D_{\boldsym{\tau}}\partial_{x}U'_{\mathbf{z},t,\boldsymbol{\tau}}(1)&=D_{\boldsym{\tau}} \partial_{y}P'_{\mathbf{z},t,\boldsymbol{\tau}}(1)\\
\nonumber
&= Dh\big(U'_{\mathbf{z},t,\boldsymbol{\tau}}(1)\big) D_{\boldsym{\tau}}\partial_{\theta} U'_{\mathbf{z},t,\boldsymbol{\tau}}(1)+ D^2h\big(U'_{\mathbf{z},t,\boldsymbol{\tau}}(1)\big) \big\{D_{\boldsym{\tau}} U'_{\mathbf{z},t,\boldsymbol{\tau}}(1), \partial_{\theta} U'_{\mathbf{z},t,\boldsymbol{\tau}}(1)\big\}.
\end{align}
On the other hand, differentiating (\ref{Bishoptype2}) with respect to $\boldsym{\tau}$ gives
$$D_{\boldsym{\tau}} U'_{\mathbf{z},t,\boldsymbol{\tau}}(\xi)= - \mathcal{T}_1\big(D_{\boldsym{\tau}} h(U'_{\mathbf{z},t,\boldsymbol{\tau}}) \big)(\xi) - \mathcal{T}_1 D_{\boldsym{\tau}} u'_{\mathbf{z},t,\boldsymbol{\tau}}(\xi).$$
In particular, this implies that  
\begin{align} \label{eq_daohamboldtautai126/12}
D_{\boldsym{\tau}} U'_{\mathbf{z},t,\boldsymbol{\tau}}(1)=0.
\end{align}
A similar equation for $D_{\mathbf{z}} U'_{\mathbf{z},t,\boldsymbol{\tau}}(\xi)$ and (\ref{eq_daohamboldtautai126/12}) shows that $D_{\mathbf{z}} U'_{\mathbf{z},t,\boldsymbol{\tau}}(1)= O(t)$ (note that $D_{\mathbf{z}} |\mathbf{z}|= O(1)$).   Now using the same reason as above, (\ref{eq_daohamboldtau26/12lauqua}) and (\ref{eq_daohamboldtautai126/12})  implies the second desired inequality. The third one is proved in the same way. The proof is finished.
 \end{proof}

\begin{lemma}\label{le_daohamcuaUtheoxytai1} There exist a positive constant $t_2<t_1$ independent of $(p_0,\mathbf{z},\boldsymbol{\tau},t)$ and  a $\mathcal{C}^1$ function 
$$\boldsymbol{\tau}(\mathbf{z},t): B^*_{2n}(0,\frac{1}{2n}) \times (0,t_2) \rightarrow B_n(0,1)$$
 so that for any $(\mathbf{z},t) \in  B^*_{2n}(0,\frac{1}{2n}) \times (0,t_2),$ we have 
\begin{align}\label{ine_daohamtheothetacuaUtai1}
\partial_{\theta}U'_{\mathbf{z},t,\boldsymbol{\tau}(\mathbf{z},t)}(e^{i\theta})|_{\theta=0}=  \frac{2t \Im \mathbf{z}}{\sqrt{|\mathbf{z}|}(2+ \sqrt{|\mathbf{z}|})}
\end{align}
and
\begin{align}\label{ine_daohamtheothetacuaUtai12}
\bigg| \partial^2_{\theta}U'_{\mathbf{z},t, \boldsymbol{\tau}(\mathbf{z},t)}(e^{i\theta})|_{\theta=0}- \frac{2t(2|\mathbf{z}|- \Re \mathbf{z})}{|\mathbf{z}|} \bigg| \le  c_2 t^2.
\end{align}
\end{lemma}

\begin{proof} The Cauchy-Riemann equations for $F'^h_{\boldsymbol{\tau}}$ give 
\begin{align} \label{eq_CauchyRiemequation}
\partial_{y}U'_{\mathbf{z},t,\boldsymbol{\tau}}(z)=-\partial_{x}P'_{\mathbf{z},t,\boldsymbol{\tau}}(z)- \partial_{x}u'_{\mathbf{z},t,\boldsymbol{\tau}}(z).
\end{align}
Combining this with (\ref{eq_daohamvhailantheotheta''}) gives
\begin{align} \label{eq_daohamtheoxcuaU'tai1tau}
\partial_{\theta}U'_{\mathbf{z},t,\boldsymbol{\tau}}(1)= \partial_{y}U'_{\mathbf{z},t}(1)=-\partial_{x}P'_{\mathbf{z},t,\boldsymbol{\tau}}(1)- \partial_{x}u'_{\mathbf{z},t,\boldsymbol{\tau}}(1).
\end{align}
Fix $\mathbf{z}$ and $t.$ Define $\Phi_0(\boldsymbol{\tau}):=\partial_{\theta}U'_{\mathbf{z},t,\boldsymbol{\tau}}(1).$ By definition of $u'_{\mathbf{z},t,\boldsymbol{\tau}}(1)$ and (\ref{eq_daohamtheoxycuau'zt}), we have 
\begin{align} \label{eq_taylorexpansioncuaPhitaugiatritai0}
\Phi_0(0)= -\partial_{x}P'_{\mathbf{z},t,0}(1)+ \frac{2t \Im \mathbf{z}}{\sqrt{|\mathbf{z}|}(2+ \sqrt{|\mathbf{z}|})}\cdot
\end{align}
The first inequality of (\ref{ine_danhgiachuancuaP'2}) implies that 
\begin{align}\label{ine_danhgiagiatri¨Phi0tai0}
\big|\Phi_0(0)- \frac{2t \Im \mathbf{z}}{\sqrt{|\mathbf{z}|}(2+ \sqrt{|\mathbf{z}|})} \big| \le c_2 t^{2}  \le t^{3/2},
\end{align}
for $t$ small enough. Differentiating (\ref{eq_daohamtheoxcuaU'tai1tau}) with respect to $\boldsymbol{\tau}$ gives 
 $$D^j_{\boldsymbol{\tau}} \Phi_0(\boldsymbol{\tau})=- D^j_{\boldsymbol{\tau}}\partial_{x}P'_{\mathbf{z},t,\boldsymbol{\tau}}(1)- D^j_{\boldsymbol{\tau}}\partial_{x}u'_{\mathbf{z},t,\boldsymbol{\tau}}(1),$$
 for $j=1$ or $2.$  By (\ref{eq_daohamtheotautai0cuau'tau}) and (\ref{ine_danhgiachuancuaP'2}), we see that  
 \begin{align}\label{ine_daohamcuaPhi0theotaubac2}
 \|D^2_{\boldsymbol{\tau}} \Phi_0\|_0 \le c_2 t^2 \le t^{3/2},
 \end{align}
  and
\begin{align}\label{ine_daohamcuaPhi0theotau}
 \frac{1}{c t} \ge \big|[D_{\boldsymbol{\tau}} \Phi_0(0)]^{-1}\big|_{\ln} \ge c [10t+ c_2 t^2]^{-1} \ge \frac{c}{11 t},
\end{align}
for $t$ small enough and some constant $c>0$ independent of $(\mathbf{z},t, \boldsymbol{\tau})$, where we recall that the norm $|\cdot|_{\ln}$ of a square matrix is the one of its associated linear map.  Taylor's expansion for $\Phi_0$ at $\boldsymbol{\tau}=0$ gives
\begin{align}\label{eq_taylorexpansioncuaPhitau}
 \Phi_0(\boldsymbol{\tau})=  \Phi_0(0)+D_{\boldsymbol{\tau}} \Phi_0(0)\boldsymbol{\tau}+ g_0(\boldsymbol{\tau}),
\end{align}
where $g_0(\boldsymbol{\tau})$ is $t^{3/2}$-Lipschitz by (\ref{ine_daohamcuaPhi0theotaubac2}) and $g_0(0)=0$.  A direct application of Lemma \ref{le_implicitfunction2} to $\Phi_0$ with $A=D_{\boldsymbol{\tau}}\Phi_0(0)$ and $M=t^{3/2}$  implies that for $t$ small enough, $\Phi_0$ is an injection on $B_n(0,1)$ and
$$\Phi\big(B_n(0,1) \big) \supset B_n\big(\Phi_0(0), ct \big)$$
Note that when $t$ is small, we see that 
$$\frac{2t \Im \mathbf{z}}{\sqrt{|\mathbf{z}|}(2+ \sqrt{|\mathbf{z}|})} \in B_n \big(\Phi_0(0), ct  \big)$$
 thanks to (\ref{ine_danhgiagiatri¨Phi0tai0}).  This yields that  there exists a unique $\boldsymbol{\tau}(\mathbf{z},t) \in B_n(0,1)$ such that   
$$\Phi_0\big(\boldsymbol{\tau}(\mathbf{z},t)\big)=-\frac{2t \Im \mathbf{z}}{\sqrt{|\mathbf{z}|}(2+ \sqrt{|\mathbf{z}|})}\cdot$$
The differentiability of $\boldsymbol{\tau}(\mathbf{z},t)$ is implied directly from the implicit function theorem for $\Phi_0(\boldsymbol{\tau},\mathbf{z},t),$ where we recovered the variable $(\mathbf{z},t)$ to indicate the dependence of $\Phi_0$ on them.  By definition of $\Phi_0,$  (\ref{ine_daohamtheothetacuaUtai1}) follows. 

Recall that $u'_{\mathbf{z},t, \boldsymbol{\tau}}= u'_{\mathbf{z},t}+  t \boldsym{\tau} \cdot \tilde{u}$ and $\partial_x \partial_y \tilde{u}(1)=0.$ Now differentiating (\ref{eq_CauchyRiemequation}) with respect to $y$ and using  (\ref{eq_daohamvhailantheotheta''}) and Lemma \ref{le_daohamtheoxU'tauzt} yield 
\begin{align*}
\partial^2_{\theta}U'_{\mathbf{z},t,\boldsymbol{\tau}}(1)&= \partial^2_{y}U'_{\mathbf{z},t,\boldsymbol{\tau}}(1)+ O(t^2)= -\partial_y \partial_{x}P'_{\mathbf{z},t,\boldsymbol{\tau}}(1) - \partial_y \partial_{x}u'_{\mathbf{z},t,\boldsymbol{\tau}}(1) + O(t^2)  \\
&= -\partial_y \partial_{x}P'_{\mathbf{z},t,\boldsymbol{\tau}}(1)+\frac{2t(2|\mathbf{z}|- \Re \mathbf{z})}{|\mathbf{z}|}+ O(t^2) \quad \text{(by (\ref{eq_daohamtheoxycuau'zt}))}.
\end{align*}
 This combined with (\ref{ine_danhgiachuancuaP'2})  implies (\ref{ine_daohamtheothetacuaUtai12}). The proof is finished.
\end{proof}

\begin{corollary} \label{cor_F'hhalfattachedtoK}
There exists a positive constant $t_3<t_2$ so that for any $t \in (0,t_3)$ we can find a positive number $\theta_t$ such that such that for  any $\mathbf{z} \in B^*_{2n}(0,\frac{1}{2n})$ the analytic disc $F'^h_{\boldsymbol{\tau}(\mathbf{z},t)}(\cdot, \mathbf{z},t)$ is $[e^{-i\theta_t},e^{i\theta_t}]$-attached to $K.$ 
\end{corollary}
\begin{proof} Write  $U'_{\mathbf{z},t,\boldsymbol{\tau}}=(U'_{\mathbf{z},t,\boldsymbol{\tau};1} ,\cdots,U'_{\mathbf{z},t,\boldsymbol{\tau};n}).$  We only need to prove that $U'_{\mathbf{z},t,\boldsymbol{\tau}(\mathbf{z},t);j}(e^{i\theta}) \ge 0$ for $|\theta|$ small enough and $1 \le j \le n$. Fix $1 \le j \le n.$ Put 
$$(\tilde{z},\delta, \gamma):=(z_j, \sqrt{|\mathbf{z}|},2|\mathbf{z}|)$$
which satisfies the condition (\ref{ine_conditionsondeltagamma}). We will mimic the proof of Lemma \ref{le_analyticdiscuz0deltagamma}.   By Lemma \ref{le_daohamcuaUtheoxytai1} and Taylor's expansion of $U'_{\mathbf{z},t,\boldsymbol{\tau};j}(e^{i\theta})$ at $\theta=0$, we have
$$t^{-1} U'_{\mathbf{z},t,\boldsymbol{\tau}(\mathbf{z},t);j} \ge  \gamma+ \frac{2\Im \tilde{z}}{\delta(2+\delta)}\theta + \frac{2(\gamma- \Re \tilde{z})}{\delta^2} \theta^2- c_2 t \theta^2 + \theta^{3} g'_{\mathbf{z}}(\theta),$$
where $g'_{\mathbf{z}}(\theta)$ is a function on $[-\pi/2,\pi/2]$ whose supnorm is bounded by $ t^{-1}\| U'_{\mathbf{z},t,\boldsymbol{\tau}(\mathbf{z},t)}\|_{3} \le 4 c_0 $ by (\ref{ine_danhgiachuancuaU'2}).  Put 
$$\tilde{f}_1(\theta):=\gamma+ \frac{2\Im \tilde{z}}{\delta(2+\delta)}\theta + \frac{\gamma- \Re \tilde{z}}{\delta^2} \theta^2$$
and 
$$\tilde{f}_2(\theta):= \frac{\gamma- \Re \tilde{z}}{\delta^2} \theta^2- c_2 t \theta^2 + \theta^{3} g'_{\mathbf{z}}(\theta).$$
We have 
$$t^{-1} U'_{\mathbf{z},t,\boldsymbol{\tau}(\mathbf{z},t);j}= \tilde{f}_1(\theta)+ \tilde{f}_2(\theta).$$
 Arguing as in Lemma \ref{le_analyticdiscuz0deltagamma}
shows that $\tilde{f}_1(\theta) \ge 0$ for all $\theta \in [-\pi/2, \pi/2]$ and $\tilde{f}_2(\theta) \ge 0$ provided that $t$ is small enough and $|\theta| \le \theta_t$ for some $\theta_t>0$ independent of $\mathbf{z}.$ Hence, $U'_{\mathbf{z},t,\boldsymbol{\tau}(\mathbf{z},t);j}(e^{i\theta}) \ge 0$ for all $1\le j \le n$ and $t$ small enough and $|\theta| \le \theta_t.$
The proof is finished. 
\end{proof}

We will need the following estimates on the function $\boldsymbol{\tau}(\mathbf{z},t).$

\begin{lemma} \label{le_estimateonboldsymboltau} Let $t, \mathbf{z}$ and $\boldsymbol{\tau}(\mathbf{z},t)$ be as in Lemma \ref{le_daohamcuaUtheoxytai1}. Then, there exist  positive constants $c_3$ and $t_4<t_3$ which are both independent of $(p_0,\mathbf{z},t)$ so that  for any $(\mathbf{z},t) \in  B^*_{2n}(0, \frac{1}{2n}) \times (0,t_4),$ we have 
\begin{align} \label{ine_chuancuatautheozt}
|\boldsymbol{\tau}(\mathbf{z},t)| \le c_3 t \quad \text{and} \quad |D_{\mathbf{z}}\boldsymbol{\tau}(\mathbf{z},t)| \le c_3 t |\mathbf{z}|^{-1}.
\end{align}
\end{lemma}

\begin{proof} We reuse the notation in the proof of in Lemma \ref{le_daohamcuaUtheoxytai1}. Recall that $\Phi_0(\boldsymbol{\tau}, \mathbf{z},t)=\partial_{\theta}U'_{\mathbf{z},t,\boldsymbol{\tau}}(1).$ Thus, (\ref{ine_danhgiachuancuaU'2}) implies that
\begin{align}\label{ine_daohamcuaPhitautheoz}
|D_{\mathbf{z}} D_{\boldsym{\tau}}\Phi_0(\boldsymbol{\tau}, \mathbf{z},t)| \le  4 c_0 t |\mathbf{z}|^{-1}.
\end{align}
 Since 
\begin{align*} 
\Phi_0\big(\boldsymbol{\tau}(\mathbf{z},t), \mathbf{z},t\big)=\frac{2t \Im \mathbf{z}}{\sqrt{|\mathbf{z}|}(2+ \sqrt{|\mathbf{z}|})},
\end{align*}
using (\ref{eq_taylorexpansioncuaPhitau}) and   (\ref{eq_taylorexpansioncuaPhitaugiatritai0}), we have 
\begin{align} \label{eq_danhgiachuanC0cuatau}
- \partial_x P'_{\mathbf{z},t,0}(1)= D_{\boldsymbol{\tau}} \Phi_0(0,\mathbf{z},t)\boldsymbol{\tau}(\mathbf{z},t)+ g_0\big(\boldsymbol{\tau}(\mathbf{z},t)\big).
\end{align}
Since $g_0$ is $c_2t^2$-Lipschitz (see (\ref{ine_daohamcuaPhi0theotaubac2})) and $g_0(0)=0,$ we deduce from (\ref{ine_daohamcuaPhi0theotau}) that 
$$ |\partial_x P'_{\mathbf{z},t,0}(1)| \ge  c t |\boldsymbol{\tau}(\mathbf{z},t)| - c_2 t^2  |\boldsymbol{\tau}(\mathbf{z},t)| \ge t |\boldsymbol{\tau}(\mathbf{z},t)|(c - c_2 t)  .$$
This combined with (\ref{ine_danhgiachuancuaP'2}) implies that 
\begin{align} \label{ine_danhgiachuanC0Cuatauproof}
 |\boldsymbol{\tau}(\mathbf{z},t)| \le  \frac{t}{c - c_2 t} \lesssim t,
\end{align}
for $t$ small enough. Hence, the first inequality of (\ref{ine_chuancuatautheozt}) follows. 

We now prove the second one. 
Differentiating the equality (\ref{eq_danhgiachuanC0cuatau}) with respect to $\mathbf{z}$ and using  (\ref{ine_danhgiachuancuaP'2theoz}) give 
\begin{align*} 
\big[ D_{\boldsym{\tau}}\Phi_0\big(0, \mathbf{z},t\big)+D_{\boldsym{\tau}}g_0\big(\boldsymbol{\tau}(\mathbf{z},t)\big) \big] D_{\mathbf{z}}\boldsymbol{\tau}(\mathbf{z},t)+ D_{\mathbf{z}} D_{\boldsymbol{\tau}}\Phi_0\big(\boldsymbol{\tau}(\mathbf{z},t), \mathbf{z},t\big) \boldsymbol{\tau}(\mathbf{z},t)= O( t^2 |\mathbf{z}|^{-1}).
\end{align*}
This together with (\ref{ine_danhgiachuanC0Cuatauproof}) and (\ref{ine_daohamcuaPhitautheoz})  yields that
\begin{align*} 
\big[ D_{\boldsym{\tau}}\Phi_0\big(0, \mathbf{z},t\big)+D_{\boldsym{\tau}}g_0\big(\boldsymbol{\tau}(\mathbf{z},t)\big) \big] D_{\mathbf{z}}\boldsymbol{\tau}(\mathbf{z},t)=O( t^2 |\mathbf{z}|^{-1}).
\end{align*}
Multiplying the two sides of the last equality by $D_{\boldsym{\tau}}\Phi_0\big(0, \mathbf{z},t\big)^{-1}$ and using (\ref{ine_daohamcuaPhi0theotau})  and $$|D_{\boldsym{\tau}}g_0\big(\boldsymbol{\tau}(\mathbf{z},t)\big)|_{\ln}=O(t^2)\quad \text{(by (\ref{ine_daohamcuaPhi0theotaubac2}))},$$
we get 
$$|D_{\mathbf{z}}\boldsymbol{\tau}|_{\ln} \lesssim t^{-1} t^2 |\mathbf{z}|^{-1} \lesssim  t |\mathbf{z}|^{-1}.$$
The proof is finished.
\end{proof}

Let $t \in (0,t_4).$ Define the map 
$$\Phi'^h: B_{2n}(0,\frac{1}{2n}) \rightarrow \C^n$$
 by putting $\Phi'^h(0)=0$  and  $\Phi'^h(\mathbf{z})= F'^h_{\boldsymbol{\tau}(\mathbf{z},t)}(z^*, \mathbf{z},t)$ for $\mathbf{z}\not = 0,$ where $z^*:=1-\sqrt{|\mathbf{z}|}.$ Our goal is to obtain similar estimates for $\Phi'^h$ as in Lemma \ref{le_differentiabilityofPhi}. However, due to the presence of $\boldsym{\tau},$ direct comparisons between $\Phi'^h$ and $\Phi'$ do not work efficiently as in the case without singularity. In order to get the expected results, we will use the technique in Corollary  \ref{cor_khaitrientaylorcuaftai1}.  

\begin{lemma} \label{le_differentiabilityofPhi'} There is a positive constant $c_4$ independent of $\mathbf{z},p_0$ and $t$ such that  for all $\mathbf{z} \in B^*_{2n}(0,\frac{1}{2n}),$   we have 
\begin{align}\label{ine_estimateonphi'}
|\Phi'^h(\mathbf{z})- t \mathbf{z}| \le  c_4 |\mathbf{z}|(t^2+ t \sqrt{|\mathbf{z}|}),
\end{align}
and 
\begin{align} \label{ine_estimateonphi2'}
\big|D_{\mathbf{z}}\Phi'^h(\mathbf{z})- t \Id \big| \le c_4 (t^2+ t\sqrt{|\mathbf{z}|}).
\end{align}
\end{lemma} 
\begin{proof}  
We want to study the behavior of $F'_{\boldsym{\tau}(\mathbf{z},t)}(z)$ near $z=1.$  By using Taylor's expansions, it is sufficient to estimate its partial derivatives at $1.$ Put 
$$\tilde{F}'(z,\mathbf{z},t):=F'_{\boldsym{\tau}(\mathbf{z},t)}(z,\mathbf{z},t).$$
Differentiating the last equality and using the second inequality of  (\ref{ine_chuancuatautheozt}) and (\ref{ine_danhgiachuancuaU'2}), one has
\begin{align} \label{ine_tildeF'daohamtheoz}
\|D_{\mathbf{z}} \tilde{F}'(\cdot,\mathbf{z},t)\|_{3} \lesssim \|D_{\mathbf{z}}F'_{\boldsym{\tau}(\mathbf{z},t)}(\cdot,\mathbf{z},t)\|_3+\|D_{\boldsym{\tau}}F'_{\boldsym{\tau}(\mathbf{z},t)}(z,\mathbf{z},t)\|_3 \,|D_{\mathbf{z}} \boldsym{\tau}(\mathbf{z},t)| \lesssim t |\mathbf{z}|^{-1}. 
\end{align}
Lemma  \ref{le_daohamcuaUtheoxytai1} implies that
$$\partial_x\Im \tilde{F}'(1,\mathbf{z},t)=- \partial_y U'_{\mathbf{z},t, \boldsym{\tau}(\mathbf{z},t)}(1)=- \frac{2t \Im \mathbf{z}}{\sqrt{|\mathbf{z}|}(2+ \sqrt{|\mathbf{z}|})}.$$
On the other hand, we have
$$\Im \tilde{F}'(z,\mathbf{z},t) = P'_{\mathbf{z},t, \boldsymbol{\tau}(\mathbf{z},t)}(z)+ u'_{\mathbf{z},t, \boldsymbol{\tau}(\mathbf{z},t)}(z).$$
Hence, 
\begin{align*}
\partial_x^2 \Im \tilde{F}'(1,\mathbf{z},t)&=\partial_x^2 P'_{\mathbf{z},t, \boldsymbol{\tau}(\mathbf{z},t)}(1)+    \partial_x^2 u'_{\mathbf{z},t, \boldsymbol{\tau}(\mathbf{z},t)}(1)\\
&=\partial_x^2 P'_{\mathbf{z},t, \boldsymbol{\tau}(\mathbf{z},t)}(1)+    \partial_x^2 u'_{\mathbf{z},t}(1)+ t \boldsym{\tau}(\mathbf{z},t) \cdot \partial_x^2\tilde{u}(1)\\
&=\partial_x^2 P'_{\mathbf{z},t, \boldsymbol{\tau}(\mathbf{z},t)}(1)+\frac{2t \Im \mathbf{z}}{\sqrt{|\mathbf{z}|}(2+ \sqrt{|\mathbf{z}|})}+t \boldsym{\tau}(\mathbf{z},t) \cdot \partial_x^2\tilde{u}(1),
\end{align*}
by (\ref{eq_daohamtheoxycuau'zt}).  By Taylor's expansion for $\Im \tilde{F}'(\cdot,\mathbf{z},t)$ at $z=1$ up to the order $3$ and using (\ref{ine_danhgiachuancuaP'2}) and the first inequality of (\ref{ine_chuancuatautheozt}), there is a function $g''_{\mathbf{z},t;1}$ defined on $[0,1]$ so that $g''_{\mathbf{z},t;1}(s)$ is $\mathcal{C}^{1,\frac{1}{2}}$ in $(s, \mathbf{z})$ and for any $s\in [0,1],$ we have  
\begin{align} \label{eq_taylorUboldsymtauzt1}
\Im \tilde{F}'(1-s,\mathbf{z},t) &= s\frac{2t \Im \mathbf{z}}{\sqrt{|\mathbf{z}|}(2+ \sqrt{|\mathbf{z}|})}+ s^2\frac{2t \Im \mathbf{z}}{2\sqrt{|\mathbf{z}|}(2+ \sqrt{|\mathbf{z}|})} \\
\nonumber
&\quad +t s^2 \boldsym{\tau}(\mathbf{z},t) \cdot \partial_x^2\tilde{u}(1)+   s^3 g''_{\mathbf{z},t;1}(s),
\end{align}
and 
$$\|g''_{\mathbf{z},t;1}\|_1 \lesssim   \| F'_{\boldsym{\tau}(\mathbf{z},t)}(\cdot,\mathbf{z},t)\|_4 \lesssim  t.$$ 
Additionally,  (\ref{ine_tildeF'daohamtheoz}) also imply that 
$$\| D_{\mathbf{z}}g''_{\mathbf{z},t;1}\|_0 \lesssim   \|D_{\mathbf{z}} \tilde{F}'(\cdot,\mathbf{z},t)\|_{3} \lesssim   t|\mathbf{z}|^{-1}.$$
Define 
$$g'_{\mathbf{z},t;1}(s):= t^{-1}sg''_{\mathbf{z},t;1}(s)+ \boldsym{\tau}(\mathbf{z},t) \cdot \partial_x^2\tilde{u}(1).$$
Thus,
$$|g'_{\mathbf{z},t;1}|_1 \lesssim s  +  t, \quad  \| D_{\mathbf{z}}g'_{\mathbf{z},t;1}\|_0 \lesssim   t|\mathbf{z}|^{-1}+  s  |\mathbf{z}|^{-1}.$$
Letting $s=\sqrt{|\mathbf{z}|}$ in (\ref{eq_taylorUboldsymtauzt1}) and using (\ref{eq_choiceofdaohamu}), (\ref{eq_choice2ofdaohamu}), we obtain
\begin{align} \label{eq_taylorUboldsymtauzt12}
\Im \tilde{F}'(1-\sqrt{|\mathbf{z}|},\mathbf{z},t)= t \Im \mathbf{z} +  t \tilde{g}'_{t;1}(\mathbf{z}),   
\end{align}
where $\tilde{g}'_{t;1}(\mathbf{z}):= |\mathbf{z}| g'_{\mathbf{z},t;1}(\sqrt{|\mathbf{z}|}).$ Direct computations give 
\begin{align} \label{eq_taylorUboldsymtauzt12g'1}
|\tilde{g}'_{t;1}(\mathbf{z})| \lesssim (t+ \sqrt{|\mathbf{z}|})|\mathbf{z}| \quad \text{and} \quad  \| D_{\mathbf{z}}\tilde{g}'_{t;1}\|_0 \lesssim t+ \sqrt{|\mathbf{z}|}.
\end{align}
Analogous arguments and Lemma \ref{le_daohamtheoxU'tauzt} also show that
$$\partial_x\Re \tilde{F}'(1,\mathbf{z},t)=\partial_x U'_{\mathbf{z},t,\boldsym{\tau}(\mathbf{z},t)}(1)=O(t^2 |\mathbf{z}|),$$
$$D_{\mathbf{z}}\partial_x\Re \tilde{F}'(1,\mathbf{z},t)=D_{\mathbf{z}}\partial_x U'_{\mathbf{z},t,\boldsym{\tau}(\mathbf{z},t)}(1)+D_{\boldsym{\tau}}\partial_x U'_{\mathbf{z},t,\boldsym{\tau}(\mathbf{z},t)}(1) D_{\mathbf{z}} \boldsym{\tau}(\mathbf{z},t)= O(t^2),$$
and 
$$\partial^2_x\Re \tilde{F}'(1,\mathbf{z},t)=- \partial^2_y\Re \tilde{F}'(1,\mathbf{z},t)=-\partial^2_{\theta} U'_{\mathbf{z},t,\boldsym{\tau}(\mathbf{z},t)}(1)+ O(t^2)=-\frac{2t(2|\mathbf{z}|- \Re \mathbf{z})}{|\mathbf{z}|}+ O(t^2)$$
by (\ref{ine_daohamtheothetacuaUtai12}). Hence, as above there exists a $\mathcal{C}^1$ function $\tilde{g}'_{t;2}(\mathbf{z})$ on $B^*_{2n}(0,\frac{1}{2n})$ such that
\begin{align} \label{eq_taylorUboldsymtauzt2}
\Re \tilde{F}'(1-\sqrt{|\mathbf{z}|},\mathbf{z},t)= t \Re \mathbf{z} +  t \tilde{g}'_{t;2}(\mathbf{z})
\end{align}
and
\begin{align} \label{eq_taylorUboldsymtauzt12g'2}
|\tilde{g}'_{t;2}(\mathbf{z})| \lesssim (t+ \sqrt{|\mathbf{z}|})|\mathbf{z}| \quad \text{and} \quad  \| D_{\mathbf{z}}\tilde{g}'_{t;2}\|_0 \lesssim t+ \sqrt{|\mathbf{z}|}.
\end{align}
By (\ref{eq_taylorUboldsymtauzt12}) and (\ref{eq_taylorUboldsymtauzt2}), we get 
$$\Phi'^h(\mathbf{z})=\tilde{F}'(1-\sqrt{|\mathbf{z}|},\mathbf{z},t)= t \mathbf{z}+ t (\tilde{g}'_{t;2}+ i \tilde{g}'_{t;1}).$$
Using (\ref{eq_taylorUboldsymtauzt12g'2}) and (\ref{eq_taylorUboldsymtauzt12g'1}), we get the desired results. The proof is finished.
\end{proof}

The proof for the following lemma is similar to Lemma \ref{le_implicittheorem}. 

\begin{lemma} \label{le_implicittheorem'} There are  positive constant $t_5<t_4$ and $r'_0<1/(2n)$ independent of $(p_0,\mathbf{z})$ such that for any $t \in (0,t_5]$ and any $\mathbf{z} \in B^*_{2n}(0,r'_0),$ the set  $\Phi'^h \big(B_{2n}(0, r'_0)\big)$ contains the ball $B_{2n}(0, r'_0t /2).$
\end{lemma}
\noindent
\begin{proof}[Proof of Proposition \ref{pro_familydiscK} for the case with singularity.] In our chosen local coordinates around $p_0,$ we have $p_0=0$ and $p=\mathbf{z}.$  Let $t \in (0,t_5]$ and $\mathbf{z} \in B^*_{2n}(0,r'_0)$ as in Lemma \ref{le_implicittheorem'}. Without loss of generality, we can suppose that 
$$t_5+ \sqrt{r'_0} \le 1/(2c_4).$$
 For any $\mathbf{z} \in B^*_{2n}(0,  r'_0 t /2),$ there exists $\mathbf{z}^* \in B_{2n}(0, r'_0)$ for which $\Phi'^h(\mathbf{z}^*)=\mathbf{z}.$ We deduce from (\ref{ine_estimateonphi'}) that 
$$|\mathbf{z}- t  \mathbf{z}^*| \le c_4 |\mathbf{z}^*|  (t^2+ t \sqrt{|\mathbf{z}|^*}) \le  \frac{|t \mathbf{z}^*|}{2}.$$
 Let $f(z):= F'^h(z,\mathbf{z}^*, t_5)$ and $z^*:= 1- \sqrt{|\mathbf{z}^*|}.$ The last inequality implies that 
$$|1-z^*|^2\le 2|\mathbf{z}|/t.$$
As in the case without singularity, the analytic disc $f$  satisfies all properties in Proposition \ref{pro_familydiscK}.  

Now, we explain how to obtain the desired analytic discs when $\dim_{\R} K >n$. In the last subsection, we sliced $K$ by generic $n$-dimensional submanifolds $K'$ in a uniform way. Then, one just applied the previous result for $K'$ to get discs partly attached to $K.$ In our present case, such slicing does not always work due to the fact that a hypersurface passing an edge of $K$ may only intersect $K$ at that point. Hence, we do not get a such a family $K'$ as above. We will use the same idea with some additional caution. As just mentioned, we only need to take care of the edges of $K.$ Let $p_e$ be an edge of $K.$ By definition of $K,$ there exists a local chart $(\tilde{W}_{p_e},\tilde{\Psi})$ of $p_e$ in $X$ such that $\tilde{\Psi}$ is a diffeomorphism from $\tilde{W}_{p_e}$ to $B_{2n}(0,2)$  and $\tilde{\Psi}(K\cap \tilde{W}_{p_e})$ is the intersection of a finite union of convex polyhedra with $B_{2n}(0,2).$ For simplicity, we identify $K$ with $\tilde{\Psi}(K)$ and suppose that $K$ is just a convex polyhedron.  Hence, it is easy to choose a $(3n-\dim K)$-dimensional subspace $H_{p_e}$ of $\R^{2n}$ such that the affine subspace $p_e+ H_{p_e}$ intersects $K$ at a $n$-dimensional convex polyhedron $K'_{p_e}$ which is generic at $p_e$ in the sense of the Cauchy-Riemann geometry: $K'_{p_e}+ J K'_{p_e}=\R^{2n}$ where $J$ is the complex structure of $X,$ we identified $T_{p_e}X$ with $R^{2n}.$ Since $p_e$ is an edge, the last property implies that the same thing also holds  for any $p_0\in \R^{2n}$ close enough to $p_e, i.e,$ $(p_0+ H_{p_e}) \cap K = K_{p_0}$ and $K_{p_0}$ generic at $p_0.$ To summarize, we just get a family of  generic $n$-dimensional local submanifolds $K'_{p_0}$ of $K$ uniformly in $p_0$. Now apply the above result for each $K_{p_0},$ we get the desired conclusion. The proof  is finished. 
\end{proof}

\bibliography{test2}
\bibliographystyle{siam}

\end{document}